\documentclass[12pt]{article}
\usepackage{amssymb,amsthm,amsmath,latexsym}
\usepackage{url}
\usepackage{color}
\usepackage{comment}
\usepackage{lscape}
\usepackage{mathtools}
\newtheorem{thm}{Theorem}[section]
\newtheorem{prop}[thm]{Proposition}
\newtheorem{lem}[thm]{Lemma}

\theoremstyle{remark}
\newtheorem{rem}[thm]{Remark}
\newcommand{\FF}{\mathbb{F}}
\newcommand{\ZZ}{\mathbb{Z}}
\newcommand{\cT}{\mathcal{DT}}
\newcommand{\cC}{\mathcal{DC}}
\newcommand{\cN}{\mathcal{DN}}
\newcommand{\0}{\mathbf{0}}

\newcommand{\ww}{\omega}
\newcommand{\vv}{\omega^2}
\DeclareMathOperator{\wt}{wt}
\DeclareMathOperator{\rank}{rank}

\begin{document}

\title{Double Toeplitz codes and their average weight enumerators}

\author{
Masaaki Harada\thanks{
Research Center for Pure and Applied Mathematics,
Graduate School of Information Sciences,
Tohoku University, Sendai 980--8579, Japan.
email: \protect\url{mharada@tohoku.ac.jp.}
}
and
Keito Yamaguchi\thanks{
Research Center for Pure and Applied Mathematics,
Graduate School of Information Sciences,
Tohoku University, Sendai 980--8579, Japan.
email: \protect\url{yamaguchi.keito.t5@dc.tohoku.ac.jp}.
}
}

\maketitle

\begin{abstract}
Recently, double Toeplitz codes have been introduced as a
generalization of double circulant codes.
In this paper, we study the average weight enumerators of double Toeplitz codes.
As an application, we consider the existence of 
double Toeplitz codes over $\FF_q$ with some specified minimum weights
for $q \in \{2,3,4\}$.
We also give a classification of double Toeplitz codes over $\FF_q$ with the largest minimum weights
for moderate lengths and $q \in \{2,3,4\}$.
\end{abstract}

\section{Introduction}

A code $C$ is called \emph{self-dual},
\emph{isodual} and \emph{formally self-dual} if 
$C=C^\perp$,
$C$ is  equivalent to $C^\perp$ and
the weight enumerator of $C$ coincides with that of $C^\perp$, respectively,
where $C^\perp$ is the dual code of $C$ under  the standard inner product.
Self-dual codes have been studied extensively for many decades
(see e.g., \cite{SPLAG} and \cite{RS-Handbook}).
Isodual codes and formally self-dual codes have also received considerable interest
(see e.g., \cite{BH}, \cite{DGH}, \cite{GH97}, \cite{HLL}, \cite{LSW}, 
\cite{LLZ}, 
\cite{SXS}, \cite{WLLL}
and the references given therein).
Double circulant codes are a remarkable class of isodual codes.

Recently, double Toeplitz codes have been introduced as
generalizations of double circulant codes in order to construct isodual
(formally self-dual) codes
with large minimum weights~\cite{SXS}.
Let $d_{q,DT}(n)$ denote the largest minimum weight among all double 
Toeplitz $[n,n/2]$ codes over $\FF_q$,
where $\FF_q$ denotes the finite field of order $q$.
It is a fundamental problem to determine the value $d_{q,DT}(n)$
for a fixed pair $(q,n)$.
The values $d_{q,DT}(n)$ were determined in~\cite{SXS} for 
moderate lengths $n$ and $q \in \{2,3,4,5,7\}$.
In addition, it was shown in~\cite{SXS} that double Toeplitz codes are asymptotically good.
Their work is the major inspiration for the present work.
Further background  on double Toeplitz codes and their generalizations can be found 
in~\cite{Cheng}, \cite{LSL}, \cite{LLZ}, \cite{LYLW}, \cite{SOXS} and \cite{WLLL}.



Weight enumerators of several types have been widely studied (see  e.g., \cite{MS-book}).
Let $\Omega_{q,n}$ be the set of all distinct double Toeplitz $[n,n/2]$ codes over $\FF_q$.
The \emph{average weight enumerator} of double Toeplitz $[n,n/2]$ codes over $\FF_q$
is defined as follows:
\[
\Psi_{q,n}(y)=\frac{1}{|\Omega_{q,n}|}\sum_{C \in \Omega_{q,n}}W_C(y),
\]
where $W_C(y)$ denotes the weight enumerator of $C$.
The average weight enumerators of self-dual codes are used to  establish the existence of
self-dual codes with specified minimum weights
(see e.g.,~\cite[Proof of Theorem~4]{C-S}).
In this paper, we study the average weight enumerators of double Toeplitz codes.
In the same way as for self-dual codes, we establish the existence of 
double Toeplitz codes over $\FF_q$ with specified minimum weights.
We say that a double Toeplitz $[n,n/2,d_{q,DT}(n)]$ code over $\FF_q$ is
\emph{DT-optimal}.
It is also a fundamental problem 
to classify DT-optimal double Toeplitz $[n,n/2]$ codes over $\FF_q$ for a fixed pair $(q,n)$.
In this paper, we give a classification of DT-optimal double Toeplitz codes over $\FF_q$
for moderate lengths and $q \in \{2,3,4\}$.

This paper is organized as follows.
In Section~\ref{Sec:2}, we provide definitions, notations and basic results.
In Section~\ref{Sec:awe}, we  theoretically derive explicit expressions of the average weight 
enumerators of double Toeplitz codes over $\FF_q$ 
without finding the set $\Omega_{q,n}$  (Theorem~\ref{thm:awe}).
As an application of Theorem~\ref{thm:awe}, in Section~\ref{Sec:app}, 
we consider the existence of double Toeplitz codes over $\FF_q$
with minimum weight at least $d$ for $q\in \{2,3,4\}$ and $d\le 10$
(Propositions~\ref{prop:minwt-F2}, \ref{prop:minwt-F3} and~\ref{prop:minwt-F4}).
In Section~\ref{Sec:classification}, we give a classification of 
DT-optimal double Toeplitz codes over $\FF_2$ and  $\FF_3$ 
for lengths up to $40$ and $26$, respectively
(Tables~\ref{Tab:F2-classification} and \ref{Tab:F3-classification}). 
We also give a classification of 
DT-optimal double Toeplitz codes over $\FF_4$
for lengths up to $20$ except $16$, 
while determining the largest minimum weights
for lengths $16, 18, \ldots, 24$ (Table~\ref{Tab:F4-classification} and 
Proposition~\ref{prop:F4-16-18-20}).
Our classification reveals the existence of 
many DT-optimal double Toeplitz codes that are inequivalent
to any of double circulant codes and double negacirculant codes.

\section{Preliminaries}\label{Sec:2}
In this section, we give definitions, notations and basic results
used throughout this paper.


\subsection{Codes and formally self-dual codes}
Let $\FF_q$ denote the finite field of order $q$,
where $q$ is a prime power.
An $[n,k]$ \emph{code $C$ over $\FF_q$}
is a $k$-dimensional vector subspace of $\FF_q^n$.
Codes over $\FF_2$, $\FF_3$ and $\FF_4$ 
are called \emph{binary}, \emph{ternary} and \emph{quaternary},
respectively.
The parameter $n$ is called the \emph{length} of $C$.
A \emph{generator matrix} of an $[n,k]$ code $C$ is a $k \times n$
matrix such that the rows of the matrix generate $C$.
Two $[n,k]$ codes $C$ and $C'$ over $\FF_q$ are
\emph{equivalent}, denoted $C \cong C'$,
if there is an $n \times n$ monomial matrix $P$ over $\FF_q$ with
$C' = CP$, where $CP=\{ x P \mid x \in C\}$.

The \emph{weight} $\wt(x)$ of a vector $x \in \FF_q^n$ is
the number of nonzero components of $x$.
A vector of a code $C$ is called a \emph{codeword} of $C$.
The minimum nonzero weight of all codewords in $C$ is called
the \emph{minimum weight} of $C$. An  $[n,k,d]$ code
is an $[n,k]$ code with minimum weight $d$.
The \emph{weight enumerator} of $C$
is given by the polynomial $W_C(y)=\sum_{c \in C}y^{\wt(c)}$.

The \emph{dual code} $C^{\perp}$  of an $[n,k]$ code $C$ over $\FF_q$
is defined as follows:
\begin{align*}
C^{\perp}&=
\{x \in \FF_q^n \mid  x \cdot y = 0 \text{ for all } y \in C\},
\end{align*}
where
$x \cdot y = \sum_{i=1}^{n} x_i {y_i}$
for $x=(x_1,x_2,\ldots,x_n)$ and $y=(y_1,y_2,\ldots,y_n) \in \FF_{q}^n$.
A generator matrix of the dual code $C^{\perp}$ of a code $C$ is called
a \emph{parity-check matrix} of $C$.
A code $C$ is called \emph{self-dual}, \emph{isodual} and
\emph{formally self-dual}
if $C=C^\perp$, $C \cong C^\perp$ and
$W_C(y)=W_{C^\perp}(y)$, respectively.
Of course, self-dual codes are automatically isodual,
and isodual codes are automatically formally self-dual.

\subsection{Double Toeplitz codes}\label{sec:DTC}

Throughout this paper, let $I_n$ denote the identity matrix of order $n$.
An $n \times n$  matrix of the following form:
\[
\left( \begin{array}{cccccc}
r_0&r_1&r_2& \cdots &r_{n-2} &r_{n-1}\\
\mu r_{n-1}&r_0&r_1& \cdots &r_{n-3}&r_{n-2} \\
\mu r_{n-2}&\mu r_{n-1}&r_0& \cdots &r_{n-4}&r_{n-3} \\
\vdots &\vdots & \vdots &&\vdots& \vdots\\
\vdots &\vdots & \vdots &&\vdots& \vdots\\
\mu r_1&\mu r_2&\mu r_3& \cdots&\mu r_{n-1}&r_0
\end{array}
\right)
\]
is called \emph{circulant} if $\mu=1$ and 
\emph{negacirculant} if $\mu=-1$.
Codes with generator matrices of the following form:
\begin{equation*}\label{eq:pDCC}
\left(\begin{array}{ccccc}
{} & I_n  & {} & A & {} \\
\end{array}\right)
\end{equation*}
are called (pure) \emph{double circulant} and \emph{double negacirculant}
if the matrices $A$ are circulant and negacirculant, respectively.
We denote the codes by 
$\cC(r)\text{ and }\cN(r)$,
respectively, where $r$ denotes the first row  of $A$.

An $n \times n$  matrix of the following form:
\[
\left( \begin{array}{cccccc}
t  &a_1&a_2& \cdots &a_{n-2} &a_{n-1}\\
b_1&t  &a_1& \cdots &a_{n-3}&a_{n-2} \\
b_2&b_1&t  & \cdots &a_{n-4}&a_{n-3} \\
\vdots &\vdots & \vdots &&\vdots& \vdots\\
\vdots &\vdots & \vdots &&\vdots& \vdots\\
b_{n-1}&b_{n-2}&b_{n-3}& \cdots&b_1&t  
\end{array}
\right)
\]
is called \emph{Toeplitz}.
We denote the above matrix by $T(t,a,b)$,
where $a=(a_1,a_2,\ldots,a_{n-1})$ and $b=(b_1,b_2,\ldots,b_{n-1})$.
Of course, $T(t,a,b)$
is circulant and negacirculant if 
\[
a_i=b_{n-i} \text{ and }a_i=-b_{n-i}\ (i=1,2,\ldots,n-1),
\]
respectively.
A code with generator matrix of the following form:
\begin{equation*}\label{q:DTC}
\left(\begin{array}{ccccc}
{} & I_n  & {} & T(t,a,b) & {} \\
\end{array}\right)
\end{equation*}
is called \emph{double Toeplitz}.  
We denote the code by 
$\cT(t,a,b)$.

Let $d_{q,DT}(n)$ denote the largest minimum weight among
all double Toeplitz $[n,n/2]$ codes over $\FF_q$.
We say that a double Toeplitz $[n,n/2,d_{q,DT}(n)]$ code over $\FF_q$ is
\emph{DT-optimal}.
The values $d_{q,DT}(n)$ were determined in~\cite{SXS} for 
$n \le 40,30,14,24,24$ and $q =2,3,4,5,7$, respectively.

The following characterizations on double Toeplitz codes
are known.

\begin{prop}[Shi, Xu and Sol\'e~\cite{SXS}]\label{prop:SXS}
\begin{enumerate}
\item
A double Toeplitz code over $\FF_q$ is isodual.
\item
A double Toeplitz self-dual code over $\FF_q$ is double circulant
or double negacirculant.
\item
A binary double Toeplitz code with only even weights is double circulant.
\end{enumerate}
\end{prop}

Further background  on double Toeplitz codes and their generalizations can be found 
in~\cite{Cheng}, \cite{LSL}, \cite{LLZ}, \cite{LYLW}, \cite{SOXS} and \cite{WLLL}.

%
%

\section{Average weight enumerators of double Toeplitz codes}\label{Sec:awe}

Throughout this paper, 
let $\Omega_{q,n}$ denote the set of all distinct double Toeplitz $[n,n/2]$ codes over $\FF_q$.
The \emph{average weight enumerator} of double Toeplitz $[n,n/2]$ codes over $\FF_q$
is defined as follows:
\[
\Psi_{q,n}(y)=\frac{1}{|\Omega_{q,n}|}\sum_{C \in \Omega_{q,n}}W_C(y).
\]
In this section, we theoretically derive explicit expressions of the average weight 
enumerators $\Psi_{q,n}(y)$ 
without finding $\Omega_{q,n}$.

Let $A^T$ denote the transpose of a matrix $A$.
Let $\0_{n}$ denote the zero vector of length $n$.
For vectors  $u$ and $v$ of $\FF_q^{n/2}$,
let $\Omega_{q,n}^{(u,v)}$ denote the set of all distinct double Toeplitz $[n,n/2]$ codes 
over $\FF_q$ containing the vector $(u, v)$.

The following important work in~\cite{SXS} was the major inspiration for this paper.

\begin{thm}[{\cite[Proposition~5, Theorem~6]{SXS}}]\label{thm:SXS}
\begin{enumerate}
\item
$|\Omega_{q,n}|=q^{n-1}$.
\item
Let $u$ and $v$ be vectors of $\FF_q^{n/2}$.
If  $\wt(u) \ne 0$, then $|\Omega_{q,n}^{(u,v)}| \le  q^{n/2}$.
\end{enumerate}
\end{thm}

By following the same line as in the proof of~\cite[Theorem~6]{SXS},
we have the following improvement.
We give a proof for the sake of completeness.

\begin{lem}\label{lem:main}
Let $u$ and $v$ be vectors  of $\FF_q^{n/2}$.  Then
\begin{equation}\label{eq:NDT}
|\Omega_{q,n}^{(u,v)}|=
\begin{cases}
q^{n-1} & \text{ if } \wt(u)=0 \text{ and } \wt(v)=0,\\
0 & \text{ if }  \wt(u)=0 \text{ and }  \wt(v)\ne0,\\
q^{n/2-1} &  \text{ if } \wt(u)\ne0.
\end{cases}
\end{equation}
\end{lem}
\begin{proof}
Let us consider a
double Toeplitz $[n,n/2]$ code $\cT(t,a,b)$ over $\FF_q$.
Note that $\cT(t,a,b)$ has generator matrix
$
\left( \begin{array}{cc}
I_{n/2}  &A\\
\end{array}
\right)
$, where
\begin{align*}
A&=
\left( \begin{array}{cccccc}
t  &a_1&a_2& \cdots &a_{n/2-2} &a_{n/2-1}\\
b_1&t  &a_1& \cdots &a_{n/2-3}&a_{n/2-2} \\
b_2&b_1&t  & \cdots &a_{n/2-4}&a_{n/2-3} \\
\vdots &\vdots & \vdots &&\vdots& \vdots\\
\vdots &\vdots & \vdots &&\vdots& \vdots\\
b_{n/2-1}&b_{n/2-2}&b_{n/2-3}& \cdots&b_1&t  
\end{array}
\right),
\end{align*}
$a=(a_1,a_2,\ldots,a_{n/2-1})$  and $b=(b_1,b_2,\ldots,b_{n/2-1})$.
Determining the number $|\Omega_{q,n}^{(u,v)}|$
is equivalent to counting the triples $(t,a,b)$
such that double Toeplitz $[n,n/2]$ codes $\cT(t,a,b)$ contain the vector $(u, v)$.

Since
$
\left( \begin{array}{cc}
-A^T & I_{n/2}\\
\end{array}
\right)
$
is a parity-check matrix of $\cT(t,a,b)$, we have that
\begin{equation}\label{eq:uv1}
(u,v)\in \cT(t,a,b)\iff A^Tu^T=v^T.
\end{equation}
Let $Q$ be the $n/2 \times n/2$ matrix defined as follows:
\[Q=
 \begin{pmatrix} 
    0 & 0 & \cdots & 0& 0 & 0 & 1 &\\
    0 & 0 & \cdots & 0& 0 & 1 & 0 &\\
    0 & 0 & \cdots & 0& 1 & 0 & 0 &\\
    \vdots &\vdots&  &  \vdots& \vdots &\vdots& \vdots\\
    \vdots &\vdots&  & \vdots &\vdots  &\vdots & \vdots\\
    1  & 0 & \cdots & 0 & 0& 0 & 0 &\\
\end{pmatrix}.
\]
Since $QA^T=AQ$, we have that
\begin{equation}\label{eq:uv12}
A^Tu^T=v^T \iff QA^Tu^T=Qv^T \iff AQu^T=Qv^T.
\end{equation}
Let $u'^T$ and $v'^T$ be defined as follows:
\begin{align*}
u'^T&=Qu^T=(u_{n/2},u_{n/2-1},\dots,u_1)^T  \text{ and } \\
v'^T&=Qv^T=(v_{n/2},v_{n/2-1},\dots,v_1)^T.
\end{align*}
From~\eqref{eq:uv1} and \eqref{eq:uv12}, we have that
\begin{equation}\label{eq:uv2}
(u,v)\in \cT(t,a,b)\iff AQu^T=Qv^T \iff Au'^T=v'^T.
\end{equation}
Let $D$  be  the $n/2\times (n-1)$ matrix defined as follows:
\[
 D=
  \begin{pmatrix} 
  0 & \cdots & \cdots & \cdots & 0 & u_{n/2} & u_{n/2-1}  & \cdots & u_1\\
  0 & \cdots & \cdots & 0 & u_{n/2} & u_{n/2-1} & \cdots & u_1 & 0\\
  0 & \cdots &  0 & u_{n/2} & u_{n/2-1} & \cdots & u_1 & 0& 0\\
  \vdots &  &  &  &  &  &  &   \vdots & \vdots\\
  \vdots &  &  &  &  &  &  &   \vdots & \vdots\\
  u_{n/2} & u_{n/2-1} & \cdots & \cdots  & u_1 & 0 & \cdots & 0 & 0\\
  \end{pmatrix}.
\]
Let $c$  be defined as follows:
\[
c=(b_{n/2-1},b_{n/2-2},\ldots,b_1,t,a_1,a_2,\ldots,a_{n/2-1}).
\]
From $Au'^T=v'^T$,
we obtain the following system of $n/2$ equations:
\begin{equation}\label{eq:Dcv}
Dc^T=v'^T.
\end{equation}

From~\eqref{eq:uv2},
$|\Omega_{q,n}^{(u,v)}|$ is the number of solutions of~\eqref{eq:Dcv}.
If $\wt(u)\ne0$, then $\rank(D)=n/2$. 
Thus, the number of solutions is $q^{n/2-1}$.
If $\wt(u)=0$, then the number of solutions is $q^{n-1}$ when $\wt(v)=0$ 
and~\eqref{eq:Dcv} has no solution when $\wt(v)\ne 0$.
This completes the proof.
\end{proof}

\begin{rem}
After completing this work, we became aware of \cite[Lemma~3.2]{WLLL}, 
where $|\Omega_{q,n}^{(u,v)}|=q^{n/2-1}$  if $\wt(u)\ne0$ is stated.
\end{rem}


The above lemma provides a key idea for the following theorem, 
which constitutes one of the main results of this paper.

\begin{thm}\label{thm:awe}
Let $\Psi_{q,n}(y)$ denote the average weight enumerator of  all distinct 
double Toeplitz $[n,n/2]$ codes over $\FF_q$.
Then
\begin{multline*}
|\Omega_{q,n}|  \Psi_{q,n}(y)=q^{n-1}+q^{\frac{n}{2}-1}\left(\sum_{j=1}^{\frac{n}{2}}\left(\binom{n}{j}-\binom{\frac{n}{2}}{j}\right)(q-1)^jy^j \right.\\
\left.+\sum_{j=\frac{n}{2}+1}^{n}\binom{n}{j}(q-1)^jy^j\right).
\end{multline*}
\end{thm}
\begin{proof}
Let $u$ and $v$ be vectors of $\mathbb{F}_q^{n/2}$.
Let $V(j)$ be defined as follows:
\[
V(j)=\{x\in \mathbb{F}_q^{n}\mid \wt(x)=j\}
\] 
and let $S_{i}(j)$ $(i=1,2,3)$ be defined as follows:
\begin{align*}
S_{1}(j)&=
\{(u,v)\in V(j)\mid \wt(u)=0 \text{ and }\wt(v)\ne0\}, \\
S_{2}(j)&=\{(u,v)\in V(j)\mid \wt(u)\ne0\text{ and } \wt(v)=0\} \text{ and } \\
S_{3}(j)&=\{(u,v)\in V(j)\mid \wt(u)\ne0\text{ and } \wt(v)\ne0\}.
\end{align*}
Then we have that 
\begin{equation}\label{eq:proof3}
V(j)=S_{1}(j)\sqcup S_{2}(j)\sqcup S_{3}(j)
\end{equation}
for $j \ge 1$.
If $1\le j\le n/2$, then we have that
\begin{equation}\label{eq:proof1}
|S_{i}(j)|=
\begin{cases}
\binom{\frac{n}{2}}{j}(q-1)^j &\text{ if } i=1,  \\
\binom{\frac{n}{2}}{j}(q-1)^j  &\text{ if } i=2, \\
\left(\binom{n}{j}-2\binom{\frac{n}{2}}{j}\right)(q-1)^j &\text{ if } i=3.
\end{cases}
\end{equation}
If $n/2+1\le j\le n$, then we have that
\begin{equation}\label{eq:proof2}
|S_{i}(j)|=
\begin{cases}
0 &\text{ if } i=1, \\
0 &\text{ if } i=2, \\
\binom{n}{j}(q-1)^j &\text{ if }i=3.
\end{cases}
\end{equation}
Note that $|\Omega_{q,n}|=q^{n-1}$ (Theorem~\ref{thm:SXS} (i)).
Hence, we have that
\allowdisplaybreaks
\begin{align*}
|\Omega_{q,n}|\Psi_{q,n}(y)
=\;&\sum_{j=0}^{n}\left(\sum_{C \in \Omega_{q,n}}|\{x\in C\mid \wt(x)=j\}|\right)y^j \\
=\;&\sum_{j=0}^{n}\left(\sum_{\substack{x\in V(j)}}|\{ {C \in \Omega_{q,n}}\mid x\in C\}|\right)y^j \\
=\;&|\{ {C \in \Omega_{q,n}}\mid \0_n \in C\}| \\
&
+\sum_{j=1}^{n}\left(
\sum_{x\in S_{1}(j)}|\{ {C \in \Omega_{q,n}}\mid x\in C\}|\right.
\\
&
+\sum_{x\in S_{2}(j)}|\{ {C \in \Omega_{q,n}}\mid x\in C\}| \\
&\left.
+\sum_{x\in S_{3}(j)}|\{ {C \in \Omega_{q,n}}\mid x\in C\}|\right)y^j 
\tag*{\text{(by~\eqref{eq:proof3})}}\\
=\;&
q^{n-1}+\sum_{j=1}^{n}\left(|S_{1}(j)|\cdot 0+|S_{2}(j)|q^{\frac{n}{2}-1}
+|S_{3}(j)|q^{\frac{n}{2}-1}\right)y^j 
\tag*{\text{(by~\eqref{eq:NDT})}}\\
=\;&q^{n-1}
+\sum_{j=1}^{\frac{n}{2}}\left(\binom{\frac{n}{2}}{j}(q-1)^jq^{\frac{n}{2}-1}\right.
\\&
\left.+\left(\binom{n}{j}-2\binom{\frac{n}{2}}{j}\right)(q-1)^jq^{\frac{n}{2}-1}\right)y^j
\tag*{\text{(by~\eqref{eq:proof1})}}\\
&+\sum_{j=\frac{n}{2}+1}^{n}\left(0\cdot q^{\frac{n}{2}-1}+\binom{n}{j}(q-1)^jq^{\frac{n}{2}-1}\right)y^j 
\tag*{\text{(by~\eqref{eq:proof2})}}\\
=\;&q^{n-1}+q^{\frac{n}{2}-1}\left(\sum_{j=1}^{\frac{n}{2}}\left(\binom{n}{j}-\binom{\frac{n}{2}}{j}\right)(q-1)^jy^j \right.\\
&\left.+\sum_{j=\frac{n}{2}+1}^{n}\binom{n}{j}(q-1)^jy^j\right).
\end{align*}
This completes the proof.
\end{proof}


As described above,
the average weight enumerators of self-dual codes are used to  establish the existence of
self-dual codes with specified minimum weights
(see e.g.,~\cite[Proof of Theorem~4]{C-S}).
In the same way, we establish the existence of double Toeplitz codes over $\FF_q$ with 
specified minimum weights as an application of Theorem~\ref{thm:awe}.
For this purpose, just as in the case of self-dual codes, 
the following proposition is required.
Although the following proposition is somewhat trivial, we give a proof for the sake of completeness.

\begin{prop}\label{prop:awe}
If we write
\[
\Psi_{q,n}(y) = \frac{1}{|\Omega_{q,n}|}\sum_{i=0}^{n}\psi_{q,n,i}y^i
\]
and we suppose that
\[
\sum_{i=1}^{d-1}\psi_{q,n,i} < (q-1)|\Omega_{q,n}|.
\]
Then there is a double Toeplitz $[n,n/2]$ code over $\FF_q$ with
minimum weight at least $d$.
\end{prop}
\begin{proof}
Let $\Omega_{q,n,d}$ be the set of all distinct double Toeplitz $[n,n/2,d]$ codes over $\FF_q$.
Let $V(j)$ be defined as follows:
\[
V(j)=\{x\in \mathbb{F}_q^{n}\mid \wt(x)=j\}.
\]
For any code $C$ over $\FF_q$, it is trivial that $x$ is a codeword of weight $d$ in $C$ if and only if
$\alpha x$ is a codeword of weight $d$ in $C$ for $\alpha \in \FF_q\setminus \{0\}$.
This gives that
\[
(q-1)|\Omega_{q,n,d}| \le \sum_{x\in V(d)}|\{C\in\Omega_{q,n}\mid x\in C\}|.
\]
Since $\sum_{i=1}^{d-1}\psi_{q,n,i} < (q-1)|\Omega_{q,n}|$, 
we have that
\begin{align*}
(q-1)|\Omega_{q,n}|&
>\sum_{i=1}^{d-1}\psi_{q,n,i}\\&
=\sum_{i=1}^{d-1}\sum_{C\in\Omega_{q,n}}|\{x\in C \mid {\mathop{\mathrm{wt}}\nolimits}(x)=i\}|\\&
=\sum_{i=1}^{d-1}\sum_{x\in V(i)}|\{C\in\Omega_{q,n}\mid x\in C\}|\\&
\ge (q-1)\sum_{i=1}^{d-1}|\Omega_{q,n,i}|.
\end{align*}
This gives that $\sum_{i=1}^{d-1}|\Omega_{q,n,i}| <  |\Omega_{q,n}|$.
Then there must be  a
double Toeplitz $[n,n/2]$ code over $\FF_q$ with
minimum weight at least $d$.
The result follows.
\end{proof}

\section{Application of average weight enumerators}\label{Sec:app}

In this section and the next section, 
we present numerical results for
double Toeplitz codes over $\FF_2$, $\FF_3$ and  $\FF_4$
obtained by computer calculations.
All computer calculations in these sections
were done using programs in \textsc{Magma}~\cite{Magma}.

In this section, 
we consider the existence of double Toeplitz codes with specified minimum weights
as an application of Theorem~\ref{thm:awe}.
In particular, we determine the lengths for which there is a double Toeplitz code
over $\FF_2$ and $\FF_3$
with minimum weight at least $d$ for $d\le 10$.
We also determine the lengths for which there is a double Toeplitz code
over $\FF_4$
with minimum weight at least $d$ for $d\le 9$.

\subsection{Preparations}

Throughout this section and the next section,
assume that $q\in \{2,3,4\}$ unless otherwise specified.
As usual, we take 
$\FF_{2}$ to be $\{0,1\}$, $\FF_{3}$ to be $\{0,1,2\}$ and
$\FF_{4}$ to be $\{0,1,\ww,\vv\}$, where $\ww^2 = \ww +1$.
In addition,
using the notations $\cC(r)$, $\cN(r)$ and $\cT(t,a,b)$,
we present constructions of double circulant codes, double negacirculant codes  and double Toeplitz codes,
respectively
(see Section~\ref{sec:DTC} for the notations).


\begin{lem}
Suppose that $q \in \{2,3,4\}$.
\begin{enumerate}
\item
The double circulant $[n,n/2]$ code $\cC((0,0,\ldots,0))$ over $\FF_q$
has minimum weight
$1$ for all positive even integers $n$.
\item
The double circulant $[n,n/2]$ code $\cC((1,0,0,\ldots,0))$ over $\FF_q$
has minimum weight
$2$ for all positive even integers $n$.
\item
The double circulant $[n,n/2]$ code $\cC((1,1,0,0,\ldots,0))$  over $\FF_q$
has minimum weight $3$ for all even integers $n$ with $n \ge 6$.
The ternary double negacirculant  $[4,2]$ code $\cN((1,1))$
has minimum weight $3$.
The quaternary double circulant $[4,2]$ code $\cC((1,\ww))$
has minimum weight $3$.
\item
The double circulant $[n,n/2]$ code $\cC((0,1,1,\ldots,1))$  over $\FF_q$ 
has minimum weight $4$  for all even integers $n$ with $n \ge 8$.
The quaternary double circulant $[6,3]$ code $\cC((1,\ww,1))$
has minimum weight $4$.
\end{enumerate}
\end{lem}
\begin{proof}
The proof is straightforward and is omitted.
\end{proof}

It is known that there is no binary $[4,2,3]$ code and
there is no $[6,3,4]$ code over $\FF_q$ for $q\in\{2,3\}$ (see~\cite{Grassl}).
Hence, we determine the lengths for which there is a double Toeplitz code
over $\FF_q$ with minimum weight $d$ for 
$q\in\{2,3,4\}$ and $d\le 4$.

\begin{table}[thbp]
\caption{Values $n_{2}(d)$, $n_{3}(d)$ and $n_{4}(d)$ for $5 \le d \le 10$}
\label{Tab:awe2-4}
\centering
\medskip
{\small
\begin{tabular}{c|c|c|c}
\noalign{\hrule height1pt}
$d$ & $n_{2}(d)$  & $n_{3}(d)$ & $n_{4}(d)$\\
\hline
 5 & 30 & 20 & 16 \\
 6 & 40 & 26 & 22 \\
 7 & 48 & 32 & 26 \\
 8 & 56 & 38 & 32 \\
 9 & 66 & 44 & 38 \\
10 & 74 & 50 & 42 \\
\noalign{\hrule height1pt}
\end{tabular}
}
\end{table}

Now we consider the existence of double Toeplitz codes over $\FF_q$
with minimum weight at least $d$ for $q\in \{2,3,4\}$ and 
$5 \le d \le 10$.
For these $q$ and $d$,
by Theorem~\ref{thm:awe} and Proposition~\ref{prop:awe},
there is a double Toeplitz $[n,n/2,d]$ code over $\FF_q$ with minimum weight at least $d$
if $n \ge n_{q}(d)$,
where the value $n_{q}(d)$ is listed in Table~\ref{Tab:awe2-4}.
Here we investigate the existence of 
double Toeplitz $[n,n/2,d]$ codes over $\FF_q$ for $n < n_{q}(d)$ explicitly.

\subsection{Binary double Toeplitz codes}

\begin{itemize}
\item $d=5$:
There is a binary double circulant $[16,8,5]$ code~\cite[Fig.~16.7]{MS-book}.
For $n=18,20,\ldots,28$, 
we verified that the binary double circulant $[n,n/2]$ code $\cC(r)$
has minimum weight $5$, where
\[
r=(1,1,1,0,1,0,0,\ldots,0).
\]

It is known that there is no binary $[n,n/2,5]$ code if $n \le 14$ (see~\cite{Grassl}).

\item $d=6$:
For $n=18,20,\ldots,28,34$,
there is a binary double circulant $[n,n/2,6]$ code~\cite{C-S}, \cite{GH97} and \cite[Fig.~16.7]{MS-book}.
For $n=30,32,36,38$, we verified that 
the binary double circulant $[n,n/2]$ code $\cC(r)$
has minimum weight $6$, where
\[
r=(1,1,1,1,0,1,0,0,\ldots,0).
\]

It is known that there is no binary $[n,n/2,6]$ code if $n \le 16$ (see~\cite{Grassl}).

\item $d=7$:
For $n=22,26$, 
there is a binary double circulant $[n,n/2,7]$ code~\cite[Fig.~16.7]{MS-book}.
We verified that 
the binary double circulant $[28,14]$ code $\cC(r)$
has minimum weight $7$, where
\begin{align*}
r=
(1,0,1,1,1,1,0,0,1,0,0,0,0,0).
\end{align*}
For $n=30,32,\ldots,46$, we verified that 
the binary double circulant $[n,n/2]$ code $\cC(r)$
has minimum weight $7$, where
\[
r=(1,1,1,0,1,1,0,1,0,0,\ldots,0).
\]
We verified that the binary double Toeplitz  $[24,12]$ code $\cT(0,a,b)$
has minimum weight $7$, where
\begin{align*}
a&=( 1, 1, 1, 1, 0, 1, 1, 0, 0, 0, 0) \text{ and }\\
b&=( 1, 0, 0, 0, 1, 1, 0, 1, 1, 1, 1).
\end{align*}

It is known that there is no binary $[n,n/2,7]$ code if $n \le 20$ (see~\cite{Grassl}).

\item $d=8$:
It is well known that the binary extended Golay $[24,12,8]$ code is constructed as double circulant
(see~\cite[Fig.~16.4]{MS-book}).
For $n=28,30,\ldots,40,44$, 
there is a binary double circulant $[n,n/2,8]$ code~\cite{C-S}, \cite{GH97}, \cite{HGK} and \cite[Fig.~16.7]{MS-book}.
For $n=42,46,48,\ldots,54$, we verified that 
the binary double circulant $[n,n/2]$ code $\cC(r)$
has minimum weight $8$, where
\[
r=(1,1,1,1,0,1,1,0,1,0,0,\ldots,0).
\]

It is known that there is no binary $[n,n/2,8]$ code if $n \le 22$ and $n=26$ (see~\cite{Grassl}).

\item $d=9$:
There is a binary double circulant $[40,20,9]$ code~\cite[Fig.~16.7]{MS-book}.
The binary double circulant $[42,21]$ code $\cC(r)$
has minimum weight $9$, where
\[
r=(1, 1, 0, 1, 0, 1, 1, 1, 1, 0, 0, 1, 0, 0, 0, 0, 0, 0, 0, 0, 0).
\]
For $n=44,46,\ldots,64$, we verified that 
the binary double circulant $[n,n/2]$ code $\cC(r)$
has minimum weight $9$, where
\[
r=(1,1,1,0,1,1,1,0,0,1,0,1,0,0,\ldots,0).
\]

It is known that there is no binary $[n,n/2,9]$ code if $n \le 36$ 
(see~\cite{Grassl}).
In addition, there is no binary double Toeplitz $[38,19,9]$ code~\cite[Table~1]{SXS}.

\item $d=10$:
For $n=42,44,\ldots,54,58$,
there is a binary double circulant $[n,n/2,10]$ code~\cite{C-S}, \cite{GH97}, \cite{HGK} and \cite[Fig.~16.7]{MS-book}.
For $n=56,60,62,\ldots,72$, we verified that 
the binary double circulant $[n,n/2]$ code $\cC(r)$
has minimum weight $10$, where
\[
r=(1,1,0,1,0,1,1,1,1,1,0,0,1,0,0,\ldots,0).
\]

It is known that there is no binary $[n,n/2,10]$ code if $n \le 38$
(see~\cite{Grassl}).
In addition, there is no binary double Toeplitz $[40,20,10]$ code~\cite[Table~1]{SXS}.

\end{itemize}

Therefore, we have the following:

\begin{prop}\label{prop:minwt-F2}
\begin{enumerate}
\item
A binary double Toeplitz $[n,n/2]$ code with minimum weight at least $5$
exists precisely for $n \ge 16$.

\item
A binary double Toeplitz $[n,n/2]$ code with minimum weight at least $6$
exists precisely for $n \ge 18$.

\item
A binary double Toeplitz $[n,n/2]$ code with minimum weight at least $7$
exists precisely for $n \ge 22$.

\item
A binary double Toeplitz $[n,n/2]$ code with minimum weight at least $8$
exists precisely for $n=24$ and $n \ge 28$.

\item
A binary double Toeplitz $[n,n/2]$ code with minimum weight at least $9$
exists precisely for $n \ge 40$.

\item
A binary double Toeplitz $[n,n/2]$ code with minimum weight at least $10$
exists precisely for $n \ge 42$.
\end{enumerate}
\end{prop}

\subsection{Ternary double Toeplitz codes}\label{sec:T}

\begin{itemize}
\item $d=5$:
For $n=10,12$,
there is a ternary double circulant $[n,n/2,5]$ code~\cite{DGH} and \cite{GH2017}.
For $n=14,16,18$, 
we verified that the ternary double circulant $[n,n/2]$ code $\cC(r)$
has minimum weight $5$, where
\[
r=(1,2,1,1,0,0,\ldots,0).
\]
The existence of a ternary double Toeplitz $[10,5,5]$ code is also known~\cite[Table~2]{SXS}.

It is known that there is no ternary $[n,n/2,5]$ code if $n \le 8$ (see~\cite{Grassl}).

\item $d=6$:
The ternary extended Golay $[12,6,6]$ code is constructed as  
the double negacirculant code $\cN(r)$, where
\[
r=(1, 2, 1, 1, 1, 0).
\]
For $n=14,16,18$,
there is a ternary double circulant $[n,n/2,6]$ code~\cite{DGH}.
For $n=20,22,24$, 
we verified that the ternary double circulant $[n,n/2]$ code $\cC(r)$
has minimum weight $6$, where
\[
r=(1,2,2,1,1,0,0,0).
\]
For $n=12,14,16,18$,
the existence of a ternary double Toeplitz $[n,n/2,6]$ code is also known~\cite[Table~2]{SXS}.

It is known that there is no ternary $[n,n/2,6]$ code if $n \le 10$ (see~\cite{Grassl}).

\item $d=7$:
There is a ternary double circulant $[20,10,7]$ code~\cite{DGH}.
For $n=22,24,\ldots,30$, 
we verified that the ternary double circulant $[n,n/2]$ code $\cC(r)$
has minimum weight $7$, where
\[
r=(1,2,1,1,2,0,1,0,0,\ldots,0).
\]
The existence of a ternary double Toeplitz $[20,10,7]$ code is also known~\cite[Table~2]{SXS}.

It is known that there is no ternary $[n,n/2,7]$ code if $n \le 18$ (see~\cite{Grassl}).

\item $d=8$:
For $n=22,24,26$,
there is a ternary double circulant $[n,n/2,8]$ code~\cite{DGH} and \cite{GH2017}.
For $n=28,30,\ldots,36$, 
we verified that the ternary double circulant $[n,n/2]$ code $\cC(r)$
has minimum weight $8$, where
\[
r=(1, 1, 1, 2, 1, 1, 0, 1, 0, 0, \ldots, 0).
\]
For $n=22,26,28$,
the existence of a ternary double Toeplitz $[n,n/2,8]$ code is also known~\cite[Table~2]{SXS}.

It is known that there is no ternary $[n,n/2,8]$ code if $n \le 20$ (see~\cite{Grassl}).

\item $d=9$:
The two inequivalent ternary self-dual $[24,12,9]$ codes are constructed as
double negacirculant~\cite{HHKK}.
For $n=28,30,32$,
there is a ternary double circulant $[n,n/2,9]$ code~\cite{DGH} and \cite{GH2017}.
For $n=34,36,\ldots,42$, 
we verified that the ternary double circulant $[n,n/2]$ code $\cC(r)$
has minimum weight $9$, where
\[
r=(1, 2, 1, 1, 1, 2, 0, 1, 1, 0, 0,\ldots,0).
\]
For $n=24,30$,
the existence of a ternary double Toeplitz $[n,n/2,9]$ code is also known~\cite[Table~2]{SXS}.

It is known
that there is no ternary $[n,n/2,9]$ code if $n \le 22$
(see~\cite{Grassl}).
In addition, there is no ternary double Toeplitz $[26,13,9]$ code~\cite[Table~2]{SXS}.

\item $d=10$:
For $n=32,36$, 
there is a ternary double circulant $[n,n/2,10]$ code~\cite{GH2017}.
We verified that the ternary double circulant $[34,17]$ code $\cC(r)$
has minimum weight $10$, where
\begin{align*}
r=(1, 1, 1, 1, 0, 1, 2, 1, 0, 1, 1, 0, 0, 0, 0,0,0).
\end{align*}
For $n=38,40$, 
we verified that the ternary double circulant $[n,n/2]$ code $\cC(r)$
has minimum weight $10$, where
\[
r=(1, 2, 2, 1, 2, 1, 0, 1, 1, 0, 1, 0, 0,\ldots,0).
\]
For $n=42,44,46,48$, 
we verified that the ternary double circulant $[n,n/2]$ code $\cC(r)$
has minimum weight $10$, where
\[
r=(1, 1, 2, 2, 1, 1, 0, 1, 1, 0, 1, 0, 0,\ldots,0).
\]

It is known
 that there is no ternary $[n,n/2,10]$ code if $n \le 26$
(see~\cite{Grassl}).
In addition, there is no ternary double Toeplitz $[n,n/2,10]$ code 
for $n=28,30$~\cite[Table~2]{SXS} (see also Proposition~\ref{prop:F3-28}).

\end{itemize}

Therefore, we have the following:

\begin{prop}\label{prop:minwt-F3}
\begin{enumerate}
\item
A ternary double Toeplitz $[n,n/2]$ code with minimum weight at least $5$
exists precisely for $n \ge 10$.
\item
A ternary double Toeplitz $[n,n/2]$ code with minimum weight at least $6$
exists precisely for $n \ge 12$.
\item
A ternary double Toeplitz $[n,n/2]$ code with minimum weight at least $7$
exists precisely for $n \ge 20$.
\item
A ternary double Toeplitz $[n,n/2]$ code with minimum weight at least $8$
exists precisely for  $n \ge 22$.
\item
A ternary double Toeplitz $[n,n/2]$ code with minimum weight at least $9$
exists precisely for  $n=24$ and $n \ge 28$.

\item
A ternary double Toeplitz $[n,n/2]$ code with minimum weight at least $10$
exists precisely for $n \ge 32$.
\end{enumerate}
\end{prop}

\subsection{Quaternary double Toeplitz codes}\label{sec:4-dmax}

\begin{itemize}
\item $d=5$:
We verified that the quaternary double circulant $[10,5]$ code $\cC(r)$
has minimum weight $5$, where
\[
r=(1, \ww, 1, \ww, 0).
\]
For $n=12,14$, 
we verified that the quaternary double circulant $[n,n/2]$ code $\cC(r)$
has minimum weight $5$, where
\[
r=(1, \ww, 1, 1, 0, 0,\ldots,0).
\]
For $n=10,12$,
the existence of a quaternary double Toeplitz $[n,n/2,5]$ code is also known~\cite[Table~5]{SXS}.

It is known that there is no quaternary $[n,n/2,5]$ code if $n \le 8$ (see~\cite{Grassl}).

\item $d=6$:
For $n=14,16,18,20$, 
we verified that the quaternary double circulant $[n,n/2]$ code $\cC(r)$
has minimum weight $6$, where
\[
r=(1, \ww, 1, 1, 1, 0, 0,\ldots,0).
\]
The existence of a quaternary double Toeplitz $[14,7,6]$ code is also known~\cite[Table~5]{SXS}.

It is known
that there is no quaternary $[n,n/2,6]$ code if $n \le 10$ (see~\cite{Grassl}).
In addition, there is no quaternary double Toeplitz $[12,6,6]$ code~\cite[Table~5]{SXS}.

\item $d=7$:
For $n=18,20,22,24$, 
we verified that the quaternary double circulant $[n,n/2]$ code $\cC(r)$
has minimum weight $7$, where
\[
r=(1, \ww^2, 1, \ww, 1, 0, 1, 0, 0,\ldots,0).
\]

It is known
that there is no quaternary $[n,n/2,7]$ code if $n \le 14$ (see~\cite{Grassl}).
In addition, there is no quaternary double Toeplitz $[16,8,7]$ code
(see Proposition~\ref{prop:F4-16-18-20}).

\item $d=8$:
For $n=20,22$, 
we verified that the quaternary double circulant $[n,n/2]$ code $\cC(r)$
has minimum weight $8$, where
\begin{align*}
r=\, 
&(1, \ww^2, 1, 1, \ww, \ww, \ww^2, \ww, 0, 0) \text{ and } 
\\&(1, \ww, \ww, \ww^2, \ww, 1, 1, 0, 0, 0, 0),
\end{align*}
respectively.
For $n=24,26,28,30$, 
we verified that the quaternary double circulant $[n,n/2]$ code $\cC(r)$
has minimum weight $8$, where
\[
r=(1, 1, \ww, 1, \ww, 1, 1,0, 0,\ldots,0).
\]

It is known
that there is no quaternary $[n,n/2,8]$ code if $n \le 16$ (see~\cite{Grassl}).
In addition, there is no quaternary double Toeplitz $[18,9,8]$ code
(see Proposition~\ref{prop:F4-16-18-20}).

\item $d=9$:
For $n=24,26$, 
we verified that the quaternary double circulant $[n,n/2]$ code $\cC(r)$
has minimum weight $9$, where
\begin{align*}
r=\, 
&(1, \ww, \ww, \ww^2, \ww, \ww^2, \ww, 1, 0, \ww, 0, 0) \text{ and } 
\\&(1, \ww^2, \ww, \ww, \ww, \ww, 1, 0, 1, 0, 0, 0, 0),
\end{align*}
respectively.
For $n=28,30,32,34,36$, 
we verified that the quaternary double circulant $[n,n/2]$ code $\cC(r)$
has minimum weight $9$, where
\[
r=(1, \ww, \ww^2, \ww, 1, 1, 1, 0, 1, 0, 0, \ldots, 0).
\]

It is known
that there is no quaternary $[n,n/2,9]$ code if $n \le 20$ 
(see~\cite{Grassl}).
In addition, there is no quaternary double Toeplitz $[22,11,9]$ code
(see Proposition~\ref{prop:F4-16-18-20}).

\item $d=10$:
For $n=28,30$, 
we verified that the quaternary double circulant $[n,n/2]$ code $\cC(r)$
has minimum weight $10$, where
\begin{align*}
r=\, 
&(1, 1, \ww, \ww, \ww^2, \ww^2, 1, \ww^2, 1, \ww, 0, \ww, 0, 0) \text{ and }  \\
&(1, \ww^2, \ww, \ww^2, 1, 1, \ww, 1, 0, 1, 0, 0, 0, 0, 0),
\end{align*}
respectively.
For $n=32,34,36,38,40$, 
we verified that the quaternary double circulant $[n,n/2]$ code $\cC(r)$
has minimum weight $10$, where
\[
r=(1, \ww^2, 1, \ww, \ww, \ww, 1, 1, 0, 1, 0, 0,\ldots,0).
\]

It is known
that there is no quaternary $[n,n/2,10]$ code if $n \le 24$ (see~\cite{Grassl}).

\end{itemize}

Therefore, we have the following:

\begin{prop}\label{prop:minwt-F4}
\begin{enumerate}
\item
A quaternary double Toeplitz $[n,n/2]$ code with minimum weight at least $5$
exists precisely for $n \ge 10$.
\item
A quaternary double Toeplitz $[n,n/2]$ code with minimum weight at least $6$
exists precisely for  $n \ge 14$.
\item
A quaternary double Toeplitz $[n,n/2]$ code with minimum weight at least $7$
exists precisely for $n \ge 18$.
\item
A quaternary double Toeplitz $[n,n/2]$ code with minimum weight at least $8$
exists precisely for $n \ge 20$.
\item
A quaternary double Toeplitz $[n,n/2]$ code with minimum weight at least $9$
exists  precisely for $n \ge 24$.
\item
A quaternary double Toeplitz $[n,n/2]$ code with minimum weight at least $10$
exists for $n \ge 28$ and 
no quaternary double Toeplitz $[n,n/2]$ code with minimum weight at least $10$
exists for $n \le 24$.
\end{enumerate}
\end{prop}

It is known that there is a quaternary $[26,13,10]$ code
(see~\cite{Grassl}),
but, our exhaustive computer search verified that 
there is no quaternary double  circulant $[26,13,10]$ code.
It is worthwhile to determine whether 
there is a
quaternary double  Toeplitz $[26,13,10]$ code or not.

\subsection{Larger minimum weights}

At the end of this section, we consider the existence of 
a double Toeplitz $[n,n/2]$ code over $\FF_q$ with minimum weight at least $d$
for $11 \le d \le 50$.
By Theorem~\ref{thm:awe} and Proposition~\ref{prop:awe},
we have the following:

\begin{prop}
Suppose that $q \in \{2,3,4\}$ and 
$11 \le d \le 50$.
Let  $n_{q}(d)$ be the value listed in Table~\ref{Tab:awe2-4-2}.
Then there is a double Toeplitz $[n,n/2]$ code over $\FF_q$ with minimum weight at least $d$
if $n \ge n_{q}(d)$.
\end{prop}

\begin{table}[thb]
\caption{Values $n_{2}(d)$, $n_{3}(d)$ and $n_{4}(d)$ for $11 \le d \le 50$}
\label{Tab:awe2-4-2}
\centering
\medskip
{\small
\begin{tabular}{c|c|c|c||c|c|c|c}
\noalign{\hrule height1pt}
$d$ & $n_{2}(d)$  & $n_{3}(d)$ & $n_{4}(d)$&
$d$ & $n_{2}(d)$  & $n_{3}(d)$ & $n_{4}(d)$\\
\hline
11&  84 &  56 &  48 &31& 264 & 180 & 152 \\
12&  92 &  62 &  52 &32& 272 & 186 & 158 \\
13& 102 &  68 &  58 &33& 282 & 194 & 162 \\
14& 110 &  76 &  64 &34& 290 & 200 & 168 \\
15& 120 &  82 &  68 &35& 300 & 206 & 174 \\
16& 128 &  88 &  74 &36& 308 & 212 & 178 \\
17& 138 &  94 &  78 &37& 318 & 218 & 184 \\
18& 146 & 100 &  84 &38& 326 & 224 & 188 \\
19& 156 & 106 &  90 &39& 336 & 230 & 194 \\
20& 164 & 112 &  94 &40& 344 & 236 & 200 \\
21& 172 & 118 & 100 &41& 354 & 244 & 204 \\
22& 182 & 124 & 104 &42& 362 & 250 & 210 \\
23& 190 & 130 & 110 &43& 372 & 256 & 216 \\
24& 200 & 138 & 116 &44& 380 & 262 & 220 \\
25& 208 & 144 & 120 &45& 390 & 268 & 226 \\
26& 218 & 150 & 126 &46& 398 & 274 & 230 \\
27& 226 & 156 & 132 &47& 408 & 280 & 236 \\
28& 236 & 162 & 136 &48& 416 & 286 & 242 \\
29& 244 & 168 & 142 &49& 426 & 294 & 246 \\
30& 254 & 174 & 146 &50& 434 & 300 & 252 \\
\noalign{\hrule height1pt}
\end{tabular}
}
\end{table}

\section{DT-optimal double Toeplitz codes}\label{Sec:classification}

In this section, we give a classification of DT-optimal double Toeplitz codes
over $\FF_2$ and $\FF_3$
for lengths up to $40$ and $26$, respectively.
We also give a classification of 
DT-optimal double Toeplitz codes over $\FF_4$
for lengths up to $20$ except $16$, 
while determining the values 
$d_{4,DT}(n)$ $(n=16,18,\ldots,24)$.
Our classification reveals the existence of 
many DT-optimal double Toeplitz codes over $\FF_2$, $\FF_3$ and $\FF_4$ 
that are inequivalent
to any of double circulant codes and double negacirculant codes.

\subsection{Preparations}


\begin{lem}\label{lem:tba}
Suppose that $q \in \{2,3,4\}$.
Let $\cT(t,a,b)$ be a double Toeplitz code over $\FF_q$.
Then
\begin{enumerate}
\item
$\cT(t,a,b) \cong \cT(\alpha t,\alpha a,\alpha b)$ for $\alpha \in \FF_q\setminus\{0\}$.
\item
$\cT(t,a,b) \cong \cT(t,b,a)$.
\end{enumerate}
\end{lem}
\begin{proof}
The assertion (i) is trivial.
The code $\cT(t,a,b)$ and its dual code have the following generator matrices:
\[
\left(\begin{array}{ccccc}
{} & I_n  & {} & T(t,a,b) & {} \\
\end{array}\right)
\text{ and }
\left(\begin{array}{ccccc}
{} & -T(t,a,b)^T  & {} &I_n  & {} \\
\end{array}\right),
\]
respectively, where $\cT(t,a,b)$ has length $2n$.
Since $\cT(t,a,b)$ is isodual, $\cT(t,a,b) \cong \cT(-t,-b,-a)$.
The assertion (ii) follows from (i).
\end{proof}

We define a map $g$ from $\FF_2$ to $\ZZ$ as follows:
$g(0)=0$ and $g(1)=1$.
For $a=(a_1,a_2,\dots,a_n) \in \FF_2^n$, we define a map $f$ from $\FF_2^n$ to $\ZZ$ 
as follows: $f(a)=\sum_{i=1}^n 2^{i-1} g(a_i)$.
Then, for $a,b \in \FF_2^n$,
we define an ordering $a \ge b$ based on the natural ordering of $\ZZ$
as $f(a) \ge f(b)$.

\begin{lem} \label{lem:reduction2}
Let $\cT(t,a,b)$ be a binary double Toeplitz code.
Then there is a binary double Toeplitz code $\cT(t',a',b')$
satisfying the following conditions:
\begin{itemize}
  \item[\rm (C1)] $\cT(t,a,b) \cong \cT(t',a',b')$.
  \item[\rm (C2)] $(t',a') \ge (t',b')$.
\end{itemize}
\end{lem}
\begin{proof}
By Lemma~\ref{lem:tba} (ii), we have that
\[
\cT(t,a,b) \cong \cT(t,b,a).
\]
If $(t',a') < (t',b')$, then we can take $\cT(t',b',a')$ instead of $\cT(t',a',b')$.
This completes the proof.
\end{proof}

In a classification of binary (DT-optimal) double Toeplitz codes,
the above lemma significantly reduces the number of candidate codes,
which need be checked for equivalences.
Due to computational speed considerations,
however, we employ the following reduction for the ternary and quaternary cases.

\begin{lem} \label{lem:reduction3-4}
Suppose that $q\in \{3,4\}$.
Let $\cT(t,a,b)$ be a double Toeplitz code over $\FF_q$.
Then there is a double Toeplitz code $\cT(t',a',b')$ over $\FF_q$
satisfying the following conditions:
\begin{itemize}
  \item[\rm (C1)] $\cT(t,a,b) \cong \cT(t',a',b')$.
  \item[\rm (C3)]  The first nonzero component of the vector $(t',a')$ is $1$.
\end{itemize}
\end{lem}
\begin{proof}
By Lemma~\ref{lem:tba} (i), we have that
\[
\cT(t,a,b) \cong \cT(\alpha t,\alpha  a,\alpha b)
\]
for $\alpha \in \FF_q\setminus\{0\}$.
If the first nonzero component of the vector $(t',a')$ is $\alpha$,
then we can take $\cT(\alpha^{-1} t',\alpha^{-1}a',\alpha^{-1}b')$ instead of $\cT(t',a',b')$.
This completes the proof.
\end{proof}

\subsection{Binary DT-optimal double Toeplitz codes}

The largest minimum weight $d_{2,DT}(n)$ among binary
double Toeplitz $[n,n/2]$ codes was determined in~\cite[Table~1]{SXS}
for $n=2,4,\ldots,40$.
For these $n$,
by exhaustive computer search,
we found all distinct binary DT-optimal double Toeplitz $[n,n/2,d_{2,DT}(n)]$ codes
satisfying the condition (C2) in Lemma~\ref{lem:reduction2}, which must be checked
further for equivalences.
After equivalence testing, we completed the classification of
binary DT-optimal double Toeplitz $[n,n/2,d_{2,DT}(n)]$ codes.
In Table~\ref{Tab:F2-classification}, we list the number $N_{2,DT}(n)$ 
of inequivalent binary DT-optimal double Toeplitz 
$[n,n/2,d_{2,DT}(n)]$ codes that are inequivalent to any of
binary DT-optimal double circulant codes.
In the table, we also list  the  number $N_{2,DC}(n)$
of  inequivalent binary DT-optimal double circulant  $[n,n/2,d_{2,DT}(n)]$ codes.
Of course, the number of inequivalent binary DT-optimal double Toeplitz 
$[n,n/2,d_{2,DT}(n)]$ codes is given by $N_{2,DT}(n)+N_{2,DC}(n)$.

\begin{table}[thb]
\caption{Binary DT-optimal double Toeplitz codes}
\label{Tab:F2-classification}
\centering
\medskip
{\small
\begin{tabular}{c|c|c|cl}
\noalign{\hrule height1pt}
$n$ & $d_{2,DT}(n)$ & $N_{2,DT}(n)$ &  \multicolumn{2}{c}{$N_{2,DC}(n)$} \\
\hline
 4& 2 & 0&   2  & \cite{BH}\\
 6& 3 & 0&   1  &  (see~\cite[Fig.~16.7]{MS-book})\\
 8& 4 & 0&   1  &  (see~\cite{Jaffe})\\
10& 4 & 0&   2  &\cite{BH}\\
12& 4 & 4&   4  & \\
14& 4 & 75&   4  & \\
16& 5 & 0&   1  &\cite{BH}\\
18& 6 & 0&   1  & (see~\cite{Jaffe}, \cite[Fig.~16.7]{MS-book})\\
20& 6 & 0&   3  & \\
22& 7 & 0&   1  &\cite{BH}\\
24& 8 & 0&   1  & (see~\cite[Fig.~16.4]{MS-book})\\
26& 7 & 2&   1  &\cite{HLL}\\
28& 8 & 0&   1  & (see~\cite{Jaffe}, \cite[Fig.~16.7]{MS-book})\\
30& 8 & 0&   5  & \\
32& 8 & 1&  30  & \\
34& 8 & 2 &  52  & \\
36& 8 & 347 & 403  & \\
38& 8 &118328 & 415  & \\
40& 9 & 231 &  15  &\cite{HLL}\\
\noalign{\hrule height1pt}
\end{tabular}
}
\end{table}

\begin{table}[thbp]
\caption{Binary DT-optimal double circulant codes}
\label{Tab:F2-DCC}
\centering
\medskip
{\footnotesize
\begin{tabular}{c|l||c|l}
\noalign{\hrule height1pt}
$n$ & \multicolumn{1}{c||}{$r$} & $n$ &  \multicolumn{1}{c}{$r$} \\
\hline
%
 4&$(1,0)$ &32&$(1,1,1,1,1,0,1,0,0,1,0,0,0,0,0,0)$ \\
 4&$(1,1)$ &32&$(1,0,1,1,1,1,1,0,0,1,0,0,0,0,0,0)$ \\
 6&$(1,1,0)$ &32&$(1,1,1,0,1,1,0,1,0,1,0,0,0,0,0,0)$ \\
 8&$(1,1,1,0)$ &32&$(1,1,1,1,0,0,1,1,0,1,0,0,0,0,0,0)$ \\
10&$(1,1,1,0,0)$ &32&$(1,1,1,0,1,0,1,1,0,1,0,0,0,0,0,0)$ \\
10&$(1,1,1,1,0)$ &32&$(1,0,1,1,1,0,1,1,0,1,0,0,0,0,0,0)$ \\
12&$(1,1,1,0,0,0)$ &32&$(1,1,0,0,1,1,1,1,0,1,0,0,0,0,0,0)$ \\
12&$(1,1,0,1,0,0)$ &32&$(1,1,0,1,1,1,0,0,1,1,0,0,0,0,0,0)$ \\
12&$(1,1,1,0,1,0)$ &32&$(1,1,1,1,1,0,1,0,0,0,1,0,0,0,0,0)$ \\
12&$(1,1,1,1,1,0)$ &32&$(1,1,0,1,1,1,1,0,0,0,1,0,0,0,0,0)$ \\
14&$(1,1,1,0,0,0,0)$ &32&$(1,1,0,1,0,1,1,1,0,0,1,0,0,0,0,0)$ \\
14&$(1,1,0,1,0,0,0)$ &32&$(1,0,1,1,0,1,1,1,0,0,1,0,0,0,0,0)$ \\
14&$(1,1,1,1,0,0,0)$ &32&$(1,0,0,1,1,1,1,1,0,0,1,0,0,0,0,0)$ \\
14&$(1,1,1,1,1,1,0)$ &32&$(1,1,1,1,0,1,0,0,1,0,1,0,0,0,0,0)$ \\
16&$(1,1,1,0,1,0,0,0)$ &32&$(1,1,0,0,1,1,1,0,1,0,1,0,0,0,0,0)$ \\
18&$(1,1,1,1,0,0,1,0,0)$ &32&$(1,1,1,1,0,0,0,1,1,0,1,0,0,0,0,0)$ \\
20&$(1,1,1,1,0,1,0,0,0,0)$ &32&$(1,1,1,0,0,1,0,1,1,0,1,0,0,0,0,0)$ \\
20&$(1,1,1,0,1,1,0,0,0,0)$ &32&$(1,0,1,1,0,1,0,1,1,0,1,0,0,0,0,0)$ \\
20&$(1,1,1,1,1,0,0,1,0,0)$ &32&$(1,1,1,0,1,1,1,1,1,0,1,0,0,0,0,0)$ \\
22&$(1,1,1,0,1,1,0,1,0,0,0)$ &32&$(1,1,1,0,0,1,1,0,0,1,1,0,0,0,0,0)$ \\
24&$(1,1,0,1,1,1,1,0,1,0,0,0)$ &32&$(1,1,1,1,1,0,1,1,0,1,1,0,0,0,0,0)$ \\
26&$(1,1,0,1,0,1,0,1,1,0,0,0,0)$ &32&$(1,0,0,1,1,1,1,1,1,0,0,1,0,0,0,0)$ \\
28&$(1,1,1,0,1,0,1,1,1,0,0,0,0,0)$ &32&$(1,1,1,1,0,1,1,0,1,1,0,1,0,0,0,0)$ \\
30&$(1,1,1,0,1,1,1,0,0,0,1,0,0,0,0)$ &32&$(1,1,1,0,1,1,0,1,1,1,0,1,0,0,0,0)$ \\
30&$(1,0,1,1,0,1,1,1,0,0,1,0,0,0,0)$ &32&$(1,0,1,1,1,1,0,1,1,1,0,1,0,0,0,0)$ \\
30&$(1,1,0,1,0,1,0,1,1,0,1,0,0,0,0)$ &32&$(1,1,0,1,1,0,1,1,1,1,0,1,0,0,0,0)$ \\
30&$(1,1,0,0,1,0,1,1,1,0,1,0,0,0,0)$ &32&$(1,1,1,1,1,1,0,0,1,0,1,1,0,0,0,0)$ \\
30&$(1,1,1,1,0,1,1,1,1,1,1,0,1,0,0)$ &32&$(1,1,1,1,1,0,0,0,1,0,0,0,1,0,0,0)$ \\
32&$(1,1,1,1,0,1,1,0,1,0,0,0,0,0,0,0)$ &32&$(1,1,1,1,1,1,0,0,1,1,0,0,1,0,0,0)$ \\
\noalign{\hrule height1pt}
\end{tabular}
}
\end{table}

For the binary DT-optimal double circulant codes $\cC(r)$ of lengths $n \le 32$, 
the rows $r$ are listed in Table~\ref{Tab:F2-DCC}.
The $4$ binary DT-optimal double Toeplitz $[12,6,4]$ codes $\cT(t,a,b)$ are constructed as:
\begin{align*}
(t,a,b)=\,&
(0,(1,1,0,1,0),(1,1,1,0,0)),
\\&
(0,(1,0,1,1,0),(1,1,1,0,0)),
\\&
(0,(0,1,1,0,1),(1,1,1,0,0))  \text{ and } 
\\&
(0,(0,1,1,0,1),(1,1,0,1,0)).
\end{align*}
For the $75$ binary DT-optimal double Toeplitz $[14,7,4]$ codes $\cT(t,a,b)$,
the triples $(t,a,b)$ are listed in Table~\ref{Tab:F2-14}.
The $2$ binary DT-optimal double Toeplitz $[26,13,7]$ codes $\cT(t,a,b)$  are constructed as:
\begin{align*}
(t,a,b)=\,&
(0,(1,0,0,1,1,0,1,1,0,0,1,0),(1,0,1,0,0,1,1,0,1,1,0,0)) \text{ and } 
\\&
(0,(0,1,0,1,0,1,0,0,1,1,0,1),(1,1,0,1,1,0,0,1,0,1,0,1)).
\end{align*}
The unique binary DT-optimal double Toeplitz $[32,16,8]$ code $\cT(t,a,b)$ is constructed as:
\begin{align*}
(t,a,b)=\,&
(1,(0,0,1,0,1,0,1,1,0,0,0,1,0,1,1),
\\&
(1,1,1,0,0,1,1,0,0,0,0,0,1,0,1)).
\end{align*}
The $2$ binary DT-optimal double Toeplitz $[34,17,8]$ codes $\cT(t,a,b)$  are constructed as:
\begin{align*}
(t,a,b)=\,&
(1,(0,0,1,0,1,0,1,1,0,0,0,1,0,1,1,1),
\\&
(1,1,1,0,0,1,1,0,0,0,0,0,1,0,1,0)) \text{ and } 
\\&
(1,(0,1,0,1,1,0,0,0,0,1,0,1,0,1,1,1),
\\&
(1,0,1,1,0,0,1,1,1,1,0,1,1,1,0,0)).
\end{align*}
The remaining binary DT-optimal double Toeplitz codes 
and double circulant codes are available at
\url{https://www.math.is.tohoku.ac.jp/~mharada/DT}.

\begin{table}[thbp]
\caption{Binary DT-optimal double Toeplitz $[14,7,4]$ codes}
\label{Tab:F2-14}
\centering
\medskip
{\footnotesize
\begin{tabular}{c|c|c||c|c|c}
\noalign{\hrule height1pt}
$t$ & $a$ & $b$ & $t$ & $a$ & $b$ \\
\hline
0 & $(1,1,0,1,0,0)$  & $(1,1,1,0,0,0)$ &0 & $(1,0,1,0,0,1)$  & $(1,1,1,0,1,0)$ \\
0 & $(1,0,1,1,0,0)$  & $(1,1,1,0,0,0)$ &0 & $(1,0,1,0,0,1)$  & $(1,0,0,1,1,0)$ \\
0 & $(0,1,1,1,0,0)$  & $(1,1,0,1,0,0)$ &0 & $(0,1,1,0,0,1)$  & $(1,1,1,1,0,0)$ \\
0 & $(1,1,1,1,0,0)$  & $(1,1,1,0,0,0)$ &0 & $(0,1,1,0,0,1)$  & $(1,1,1,1,1,0)$ \\
0 & $(1,1,1,1,0,0)$  & $(1,1,0,1,0,0)$ &0 & $(1,1,1,0,0,1)$  & $(1,1,0,1,0,0)$ \\
0 & $(1,1,0,0,1,0)$  & $(1,1,1,0,0,0)$ &0 & $(1,1,1,0,0,1)$  & $(1,1,0,0,1,0)$ \\
0 & $(1,1,0,0,1,0)$  & $(1,1,1,1,0,0)$ &0 & $(1,1,1,0,0,1)$  & $(0,1,1,0,1,0)$ \\
0 & $(1,0,1,0,1,0)$  & $(0,1,1,1,0,0)$ &0 & $(1,0,0,1,0,1)$  & $(1,1,0,0,1,0)$ \\
0 & $(0,1,1,0,1,0)$  & $(1,1,1,0,0,0)$ &0 & $(1,0,0,1,0,1)$  & $(1,1,1,0,1,0)$ \\
0 & $(0,1,1,0,1,0)$  & $(1,1,0,1,0,0)$ &0 & $(0,1,0,1,0,1)$  & $(1,1,1,0,0,0)$ \\
0 & $(0,1,1,0,1,0)$  & $(0,1,1,1,0,0)$ &0 & $(0,0,1,1,0,1)$  & $(1,1,0,1,0,0)$ \\
0 & $(0,1,1,0,1,0)$  & $(1,0,1,0,1,0)$ &0 & $(1,0,1,1,0,1)$  & $(0,0,1,1,0,1)$ \\
0 & $(0,1,1,0,1,0)$  & $(0,1,1,0,1,0)$ &0 & $(0,1,1,1,0,1)$  & $(1,0,0,1,1,0)$ \\
0 & $(1,1,1,0,1,0)$  & $(1,1,0,1,0,0)$ &0 & $(0,1,1,1,0,1)$  & $(0,1,1,1,0,1)$ \\
0 & $(1,1,1,0,1,0)$  & $(0,1,1,0,1,0)$ &0 & $(1,1,1,1,0,1)$  & $(0,0,1,1,1,0)$ \\
0 & $(1,0,0,1,1,0)$  & $(1,1,0,1,0,0)$ &0 & $(1,0,0,0,1,1)$  & $(1,1,0,1,0,0)$ \\
0 & $(1,0,0,1,1,0)$  & $(1,0,1,1,0,0)$ &0 & $(1,0,1,0,1,1)$  & $(0,1,0,1,1,0)$ \\
0 & $(1,0,0,1,1,0)$  & $(1,1,1,1,0,0)$ &0 & $(1,0,1,0,1,1)$  & $(0,1,1,1,1,0)$ \\
0 & $(0,1,0,1,1,0)$  & $(1,0,1,1,0,0)$ &0 & $(0,1,1,0,1,1)$  & $(1,1,1,0,1,0)$ \\
0 & $(0,1,0,1,1,0)$  & $(0,1,1,1,0,0)$ &0 & $(0,1,1,0,1,1)$  & $(1,0,1,1,0,1)$ \\
0 & $(0,1,0,1,1,0)$  & $(1,1,0,0,1,0)$ &0 & $(1,1,1,0,1,1)$  & $(1,0,1,1,0,0)$ \\
0 & $(0,1,0,1,1,0)$  & $(1,0,1,0,1,0)$ &0 & $(1,0,0,1,1,1)$  & $(1,0,1,0,0,1)$ \\
0 & $(1,1,0,1,1,0)$  & $(1,0,1,1,0,0)$ &0 & $(1,0,0,1,1,1)$  & $(1,1,0,0,1,1)$ \\
0 & $(1,1,0,1,1,0)$  & $(0,1,1,0,1,0)$ &0 & $(0,1,0,1,1,1)$  & $(1,0,1,0,1,1)$ \\
0 & $(0,0,1,1,1,0)$  & $(1,1,0,1,0,0)$ &0 & $(1,1,0,1,1,1)$  & $(1,1,1,0,0,0)$ \\
0 & $(0,0,1,1,1,0)$  & $(1,0,1,1,0,0)$ &0 & $(1,1,0,1,1,1)$  & $(0,1,0,1,1,0)$ \\
0 & $(0,1,1,1,1,0)$  & $(1,1,0,0,1,0)$ &0 & $(1,1,0,1,1,1)$  & $(0,1,1,1,0,1)$ \\
0 & $(0,1,1,1,1,0)$  & $(1,0,1,0,1,0)$ &0 & $(1,0,1,1,1,1)$  & $(1,1,0,1,1,0)$ \\
0 & $(0,1,1,1,1,0)$  & $(0,1,1,0,1,0)$ &0 & $(1,0,1,1,1,1)$  & $(1,1,1,0,1,1)$ \\
0 & $(0,1,1,1,1,0)$  & $(0,1,0,1,1,0)$ &0 & $(1,0,1,1,1,1)$  & $(1,1,0,1,1,1)$ \\
0 & $(0,1,1,1,1,0)$  & $(1,1,0,1,1,0)$ &0 & $(0,1,1,1,1,1)$  & $(1,0,1,0,1,0)$ \\
0 & $(1,1,1,1,1,0)$  & $(1,1,1,0,0,0)$ &1 & $(0,0,1,0,1,0)$  & $(1,1,0,0,0,0)$ \\
0 & $(1,1,1,1,1,0)$  & $(0,1,1,1,0,0)$ &1 & $(1,0,1,0,1,0)$  & $(1,0,1,0,1,0)$ \\
0 & $(1,1,1,1,1,0)$  & $(1,1,0,0,1,0)$ &1 & $(0,1,0,0,0,1)$  & $(1,0,0,1,0,0)$ \\
0 & $(1,1,0,0,0,1)$  & $(1,1,1,1,0,0)$ &1 & $(1,1,1,0,0,1)$  & $(1,0,1,1,1,0)$ \\
0 & $(1,1,0,0,0,1)$  & $(1,0,0,1,1,0)$ &1 & $(1,1,0,1,0,1)$  & $(1,0,1,0,0,0)$ \\
0 & $(1,1,0,0,0,1)$  & $(1,0,1,1,1,0)$ &1 & $(0,1,1,1,1,1)$  & $(1,1,1,0,0,0)$ \\
0 & $(1,1,0,0,0,1)$  & $(1,1,1,1,1,0)$ &&&\\
\noalign{\hrule height1pt}
\end{tabular}
}
\end{table}

\subsection{Ternary DT-optimal double Toeplitz codes}\label{Sec:ternary}

The largest minimum weight $d_{3,DT}(n)$ among ternary 
double Toeplitz $[n,n/2]$ codes was determined in~\cite[Table~2]{SXS}
for $n=4,6,\ldots,30$.
However, the value $d_{3,DT}(28)$ was unfortunately reported incorrectly 
in~\cite[Table~2]{SXS}, where it appears as $8$.
Although as described in Section~\ref{sec:T}, 
there is a ternary double circulant $[28,14,9]$ code~\cite{DGH},
for confirmation,
our exhaustive computer search verified that every ternary double circulant $[28,14,9]$ code
is equivalent to the ternary double circulant $[28,14,9]$ code $\cC(r)$, where
\[
r=(1, 1, 1, 1, 2, 1, 1, 1, 2, 0, 0, 1, 0, 0).
\]
In addition, our exhaustive computer search verified that there are
$9$ inequivalent  ternary double negacirculant $[28,14,9]$ codes  $\cN(r)$
that are inequivalent to the above unique ternary double circulant code, where
\allowdisplaybreaks
\begin{align*}
r=\,&
(1,2,0,2,1,1,2,0,2,1,0,0,0,0),
(1,2,2,2,0,2,0,2,2,1,0,0,0,0),\\&
(1,1,1,1,2,1,2,0,0,0,1,0,0,0),
(1,1,2,1,2,0,0,1,2,0,1,0,0,0),\\&
(1,2,2,2,2,0,1,1,0,0,2,0,0,0),
(1,1,2,1,2,1,1,1,1,1,2,0,0,0),\\&
(1,1,2,1,2,2,2,1,2,1,0,1,0,0),
(1,2,2,1,1,1,1,1,2,2,0,1,0,0) \text{ and}\\&
(1,2,2,2,2,2,0,1,1,2,1,1,0,0).
\end{align*}
On the other hand,
it is known that there is no ternary $[28,14,d]$ code for $d \ge 11$ (see~\cite{Grassl}).
Hence, we have that $d_{3,DT}(28) \in \{9,10\}$.
Moreover, our exhaustive computer search under 
the condition (C3) in Lemma~\ref{lem:reduction3-4} verified that there is
no ternary double Toeplitz $[28,14,10]$ code.
Therefore, we have the following:

\begin{prop}\label{prop:F3-28}
$d_{3,DT}(28)=9$.
\end{prop}

Now, by exhaustive computer search,
we found all distinct ternary DT-optimal double Toeplitz $[n,n/2,d_{3,DT}(n)]$
codes satisfying the condition (C3) in Lemma~\ref{lem:reduction3-4}, 
which must be checked further for equivalences for $n=4,6,\ldots,26$.
After equivalence testing, we completed the classification of
ternary DT-optimal double Toeplitz $[n,n/2,d_{3,DT}(n)]$ codes.
In Table~\ref{Tab:F3-classification}, we list the number $N_{3,DT}(n)$ 
of inequivalent ternary DT-optimal double Toeplitz $[n,n/2,d_{3,DT}(n)]$ codes that 
are inequivalent to any of ternary DT-optimal double circulant codes  and 
ternary DT-optimal double negacirculant codes.
In the table, we also list the number $N_{3,DC}(n)$ 
of inequivalent ternary DT-optimal double circulant  $[n,n/2,d_{3,DT}(n)]$ codes
and the number $N_{3,NC}(n)$
of inequivalent ternary DT-optimal double negacirculant  $[n,n/2,d_{3,DT}(n)]$ codes
that are inequivalent to any of ternary DT-optimal double circulant codes.
Of course, the number 
of inequivalent ternary DT-optimal double Toeplitz $[n,n/2,d_{3,DT}(n)]$ codes
is given by $N_{3,DT}(n)+N_{3,DC}(n)+N_{3,NC}(n)$.

\begin{table}[thbp]
\caption{Ternary DT-optimal double Toeplitz codes}
\label{Tab:F3-classification}
\centering
\medskip
{\small
\begin{tabular}{c|c|c|cl|c}
\noalign{\hrule height1pt}
$n$ & $d_{3,DT}(n)$ & $N_{3,DT}(n)$ & \multicolumn{2}{c|}{$N_{3,DC}(n)$}&
{$N_{3,DN}(n)$}   \\
\hline
  4&  3 & 0 &    0&\cite{DGH}&  1 \\
  6&  3 & 1 &    2&\cite{DGH}&  0 \\
  8&  4 & 0 &    3&\cite{DGH}&  0 \\
 10&  5 & 0 &   1&\cite{DGH}&  0 \\
 12&  6 & 0 &    0&\cite{DGH}&  1 \\
 14&  6 & 0 &    1&\cite{DGH}&  0 \\
 16&  6 & 104 &    7&          &  5 \\
 18&  6 &156189 &   57&          &  0 \\
 20&  7 & 27 &    5&          & 11 \\
 22&  8 & 1  &    2&          &  0 \\
 24&  9 & 0 &    0&          &  2 \\
 26&  8 & 3186&  376&          &  0 \\
\noalign{\hrule height1pt}
\end{tabular}
}
\end{table}

\begin{table}[thbp]
\caption{Ternary DT-optimal double circulant codes}
\label{Tab:F3-DCC}
\centering
\medskip
{\footnotesize
\begin{tabular}{c|l||c|l||c|l}
\noalign{\hrule height1pt}
$n$ & \multicolumn{1}{c||}{$r$} &
$n$ & \multicolumn{1}{c||}{$r$} & $n$ &  \multicolumn{1}{c}{$r$} \\
\hline
16 &$(1,2,2,1,1,0,0,0)$ &18 &$(1,1,2,1,2,2,0,0,0)$ &18 &$(1,2,2,2,2,2,1,0,0)$ \\
16 &$(1,1,1,1,0,1,0,0)$ &18 &$(1,1,1,1,0,0,1,0,0)$ &18 &$(1,2,1,2,1,0,2,0,0)$ \\
16 &$(1,1,2,2,0,1,0,0)$ &18 &$(1,2,1,1,0,0,1,0,0)$ &18 &$(1,2,2,2,1,0,2,0,0)$ \\
16 &$(1,2,1,1,1,1,0,0)$ &18 &$(1,2,2,1,0,0,1,0,0)$ &18 &$(1,1,1,2,2,0,2,0,0)$ \\
16 &$(1,2,2,2,2,1,0,0)$ &18 &$(1,1,1,2,0,0,1,0,0)$ &18 &$(1,2,2,2,2,0,2,0,0)$ \\
16 &$(1,1,2,0,1,0,1,0)$ &18 &$(1,2,1,2,0,0,1,0,0)$ &18 &$(1,2,2,1,1,1,2,0,0)$ \\
16 &$(1,1,2,1,1,0,1,0)$ &18 &$(1,1,2,2,0,0,1,0,0)$ &18 &$(1,1,1,2,1,1,2,0,0)$ \\
18 &$(1,2,1,1,1,0,0,0,0)$ &18 &$(1,2,2,2,0,0,1,0,0)$ &18 &$(1,1,1,1,2,1,2,0,0)$ \\
18 &$(1,1,2,1,1,0,0,0,0)$ &18 &$(1,1,1,1,1,0,1,0,0)$ &18 &$(1,1,2,1,1,2,2,0,0)$ \\
18 &$(1,2,2,1,1,0,0,0,0)$ &18 &$(1,2,2,1,1,0,1,0,0)$ &18 &$(1,2,1,2,1,0,1,1,0)$ \\
18 &$(1,2,2,2,1,0,0,0,0)$ &18 &$(1,1,0,2,1,0,1,0,0)$ &18 &$(1,1,1,2,2,0,1,1,0)$ \\
18 &$(1,1,1,1,0,1,0,0,0)$ &18 &$(1,0,1,2,1,0,1,0,0)$ &18 &$(1,2,1,2,2,0,1,1,0)$ \\
18 &$(1,2,1,1,0,1,0,0,0)$ &18 &$(1,1,2,2,1,0,1,0,0)$ &18 &$(1,2,1,2,2,1,1,1,0)$ \\
18 &$(1,1,2,1,0,1,0,0,0)$ &18 &$(1,2,2,2,1,0,1,0,0)$ &18 &$(1,2,2,2,1,2,1,1,0)$ \\
18 &$(1,2,2,1,0,1,0,0,0)$ &18 &$(1,2,0,2,2,0,1,0,0)$ &18 &$(1,2,2,1,2,2,1,1,0)$ \\
18 &$(1,2,1,2,0,1,0,0,0)$ &18 &$(1,2,1,1,1,1,1,0,0)$ &18 &$(1,2,1,2,2,2,1,1,0)$ \\
18 &$(1,2,2,2,0,1,0,0,0)$ &18 &$(1,1,2,1,1,1,1,0,0)$ &20 &$(1,2,1,1,0,1,1,0,0,0)$ \\
18 &$(1,2,2,0,1,1,0,0,0)$ &18 &$(1,2,1,2,1,1,1,0,0)$ &20 &$(1,1,1,0,1,1,2,0,0,0)$ \\
18 &$(1,2,0,1,1,1,0,0,0)$ &18 &$(1,1,2,2,1,1,1,0,0)$ &20 &$(1,1,1,0,2,1,2,0,0,0)$ \\
18 &$(1,2,1,1,1,1,0,0,0)$ &18 &$(1,2,1,1,2,1,1,0,0)$ &20 &$(1,0,2,1,2,1,0,1,0,0)$ \\
18 &$(1,1,2,1,1,1,0,0,0)$ &18 &$(1,2,2,2,2,1,1,0,0)$ &20 &$(1,2,2,1,2,0,1,0,1,0)$ \\
18 &$(1,1,2,2,1,1,0,0,0)$ &18 &$(1,2,1,2,0,2,1,0,0)$ &22 &$(1,1,1,2,1,2,0,0,1,0,0)$ \\
18 &$(1,2,1,0,2,1,0,0,0)$ &18 &$(1,2,2,2,1,2,1,0,0)$ &22 &$(1,2,1,1,1,2,2,2,1,2,0)$ \\
18 &$(1,2,2,1,1,2,0,0,0)$ &18 &$(1,2,2,1,2,2,1,0,0)$ &&\\
\noalign{\hrule height1pt}
\end{tabular}
}
\end{table}

\begin{table}[thbp]
\caption{Ternary DT-optimal double negacirculant codes}
\label{Tab:F3-DNC}
\centering
\medskip
{\footnotesize
\begin{tabular}{c|l||c|l}
\noalign{\hrule height1pt}
$n$ & \multicolumn{1}{c||}{$r$} & $n$ &  \multicolumn{1}{c}{$r$} \\
\hline
 4 & $(1,1)$ & 20 & $(1,2,2,2,0,1,1,0,0,0)$ \\
12 & $(1,2,1,1,1,0)$ &20 & $(1,2,2,2,0,1,2,0,0,0)$ \\
16 & $(1,1,2,1,1,0,0,0)$ &20 & $(1,2,1,0,1,1,2,0,0,0)$ \\
16 & $(1,2,2,1,1,0,0,0)$ &20 & $(1,1,1,2,1,0,0,1,0,0)$ \\
16 & $(1,2,1,2,0,1,0,0)$ &20 & $(1,0,1,2,2,1,0,1,0,0)$ \\
16 & $(1,1,1,0,1,1,0,0)$ &20 & $(1,2,2,2,2,1,0,1,0,0)$ \\
16 & $(1,2,2,2,1,1,0,0)$ &20 & $(1,0,2,2,2,2,0,1,0,0)$ \\
20 & $(1,1,2,1,1,0,1,0,0,0)$ &20 & $(1,1,2,2,0,1,1,1,0,0)$ \\
20 & $(1,2,1,1,0,1,1,0,0,0)$ &24 & $(1,1,1,1,2,2,0,1,0,1,0,0)$ \\
20 & $(1,1,1,2,0,1,1,0,0,0)$ &24 & $(1,1,1,1,2,2,1,1,2,1,2,0)$ \\
\noalign{\hrule height1pt}
\end{tabular}
}
\end{table}

\begin{table}[thbp]
\caption{Ternary DT-optimal double Toeplitz $[20,10,7]$ codes}
\label{Tab:F3-20}
\centering
\medskip
{\footnotesize
\begin{tabular}{c|c|c}
\noalign{\hrule height1pt}
$t$ & $a$ & $b$ \\
\hline
0 & $(0,0,1,1,2,1,1,0,1)$ & $(1,2,0,1,1,2,2,2,2)$ \\
0 & $(0,0,1,2,1,0,1,1,1)$ & $(1,2,2,0,2,1,2,0,0)$ \\
0 & $(0,1,2,1,1,1,1,0,0)$ & $(2,0,0,2,2,2,2,1,2)$ \\
0 & $(0,1,1,1,1,2,1,0,0)$ & $(1,0,0,2,1,2,2,2,2)$ \\
0 & $(0,1,2,1,0,1,1,1,0)$ & $(2,2,2,2,0,2,1,2,0)$ \\
0 & $(0,1,0,2,1,1,1,1,0)$ & $(2,0,1,1,1,1,2,0,1)$ \\
0 & $(0,1,0,2,1,2,1,1,0)$ & $(1,0,2,2,1,2,1,0,2)$ \\
0 & $(0,1,2,1,1,2,1,0,1)$ & $(1,0,2,0,2,1,2,2,1)$ \\
0 & $(0,1,1,1,1,2,2,0,1)$ & $(2,0,2,0,1,1,2,2,2)$ \\
0 & $(0,1,0,1,0,2,2,1,1)$ & $(2,2,2,2,1,1,0,2,0)$ \\
0 & $(0,1,0,2,0,2,2,2,1)$ & $(2,2,2,1,1,1,0,1,0)$ \\
0 & $(1,0,1,2,1,1,2,1,0)$ & $(2,0,0,2,1,2,2,1,2)$ \\
1 & $(0,0,2,2,2,1,1,0,0)$ & $(2,0,0,2,2,1,1,1,0)$ \\
1 & $(0,0,1,2,0,0,1,1,1)$ & $(1,2,2,2,0,0,1,2,0)$ \\
1 & $(0,0,1,2,0,0,2,2,1)$ & $(1,2,2,1,0,0,2,1,0)$ \\
1 & $(0,1,1,2,2,2,0,1,1)$ & $(1,0,1,1,2,2,2,1,1)$ \\
1 & $(0,1,2,1,1,2,1,1,1)$ & $(1,1,2,2,2,1,2,2,1)$ \\
1 & $(0,2,2,1,1,0,1,1,1)$ & $(1,1,0,1,0,1,0,2,2)$ \\
1 & $(1,0,1,2,1,1,2,1,1)$ & $(1,2,2,2,1,2,2,1,2)$ \\
1 & $(1,0,1,2,1,2,2,1,1)$ & $(1,1,2,2,1,1,2,1,2)$ \\
1 & $(1,0,2,0,2,0,0,2,2)$ & $(2,2,0,2,2,0,2,2,2)$ \\
1 & $(1,1,2,1,1,2,1,0,1)$ & $(2,2,2,0,2,1,2,2,1)$ \\
1 & $(1,1,0,1,2,1,1,2,1)$ & $(2,2,2,1,2,2,1,2,0)$ \\
1 & $(1,1,2,0,0,2,0,0,2)$ & $(2,2,0,0,2,2,0,1,0)$ \\
1 & $(1,1,2,2,2,0,2,1,2)$ & $(2,1,1,2,1,0,1,1,1)$ \\
1 & $(1,2,0,1,2,0,2,0,0)$ & $(2,2,0,1,1,0,1,0,0)$ \\
1 & $(1,2,1,1,2,1,0,1,0)$ & $(1,2,2,2,0,2,1,2,2)$ \\
\noalign{\hrule height1pt}
\end{tabular}
}
\end{table}

For the ternary DT-optimal double circulant codes $\cC(r)$ of lengths $n=16,18,20,22$, 
the rows $r$ are listed in Table~\ref{Tab:F3-DCC}.
For the ternary DT-optimal double negacirculant codes $\cN(r)$ of lengths $n=4,12,16,20,24$, 
the rows $r$ are listed in Table~\ref{Tab:F3-DNC}.
The unique ternary DT-optimal double Toeplitz $[6,3,3]$ code $\cT(t,a,b)$ is constructed as:
\begin{align*}
(t,a,b)=(1, (1, 0), (2, 1)).
\end{align*}
For the $27$ ternary DT-optimal double Toeplitz $[20,10,7]$ codes $\cT(t,a,b)$,
the triples $(t,a,b)$ are listed in Table~\ref{Tab:F3-20}.
The unique ternary DT-optimal double Toeplitz $[22,11,8]$ code $\cT(t,a,b)$ is constructed as:
\begin{align*}
(t,a,b)=(1,(0,1,2,1,1,2,1,1,1,2),(1,1,2,2,2,1,2,2,1,2)).
\end{align*}
The remaining ternary DT-optimal double Toeplitz codes,
double circulant codes and double negacirculant codes 
are available at
\url{https://www.math.is.tohoku.ac.jp/~mharada/DT}.

\subsection{Quaternary DT-optimal double Toeplitz codes}\label{Sec:quaternary}

The largest minimum weight $d_{4,DT}(n)$ among quaternary 
double Toeplitz $[n,n/2]$ codes was determined in~\cite[Table~5]{SXS}
for $n=4,6,\ldots,14$.
For $n=16,18,\ldots,24$, if there is a quaternary $[n,n/2,d]$ code,
then $d \le 7, d \le 8, d \le 8, d \le 9, d \le 9$, respectively (see~\cite{Grassl}).
For $(n,d)=(16,7)$, $(18,8)$, $(22,9)$,
our exhaustive computer search under 
the condition (C3) in Lemma~\ref{lem:reduction3-4} verified that there is
no quaternary double Toeplitz $[n,n/2,d]$ code.
The existence of a quaternary double circulant $[n,n/2,d]$ code is known 
for $(n,d)=(16,6)$, $(18,7)$, $(20,8)$, $(22,8)$, $(24,9)$  (see Section~\ref{sec:4-dmax}).
Therefore, we have the following:

\begin{prop}\label{prop:F4-16-18-20}
\[
d_{4,DT}(n)=
\begin{cases}
6 & \text{ if }n=16, \\
7 & \text{ if }n=18, \\
8 & \text{ if }n=20, \\
8 & \text{ if }n=22, \\
9 & \text{ if }n=24.
\end{cases}
\]
\end{prop}

For $n=4,6,\ldots,14,18,20$,
by exhaustive computer search, 
we found all distinct quaternary DT-optimal double Toeplitz $[n,n/2,d_{4,DT}(n)]$ 
codes satisfying the condition (C3) in Lemma~\ref{lem:reduction3-4}, 
which must be checked further for equivalences.
After equivalence testing, we completed the classification of
quaternary DT-optimal double Toeplitz $[n,n/2,d_{4,DT}(n)]$ codes.
In Table~\ref{Tab:F4-classification}, we list the number $N_{4,DT}(n)$ 
of  inequivalent quaternary  DT-optimal double Toeplitz 
$[n,n/2,d_{4,DT}(n)]$ codes that are inequivalent to any of
quaternary  DT-optimal double circulant codes.
In the table, we also list the number $N_{4,DC}(n)$
of  inequivalent quaternary  DT-optimal double circulant  $[n,n/2,d_{4,DT}(n)]$ codes.
Of course, the number of inequivalent quaternary  DT-optimal double Toeplitz 
$[n,n/2,d_{4,DT}(n)]$ codes is given by  $N_{4,DT}(n)+N_{4,DC}(n)$.
Due to limitations on computing time, we were unable to complete a classification 
of quaternary DT-optimal double Toeplitz $[16,8,6]$ codes.
The primary difficulty arises from the sheer number of such codes.

\begin{table}[thbp]
\caption{Quaternary DT-optimal double Toeplitz codes}
\label{Tab:F4-classification}
\centering
\medskip
{\small
\begin{tabular}{c|cl|c|c}
\noalign{\hrule height1pt}
$n$&\multicolumn{2}{c|}{$d_{4,DT}(n)$} & $N_{4,DT}(n)$ & $N_{4,DC}(n)$ \\
\hline
 4 & 3 &~\cite[Table~5]{SXS} & 0&  1 \\
 6 & 4 &~\cite[Table~5]{SXS} & 0&  1 \\
 8 & 4 &~\cite[Table~5]{SXS} & 7&  6 \\
10 & 5 &~\cite[Table~5]{SXS} & 2&  2 \\
12 & 5 &~\cite[Table~5]{SXS} & 6864&  13 \\
14 & 6 &~\cite[Table~5]{SXS} & 360&  19 \\
16 & 6 & Proposition~\ref{prop:F4-16-18-20} & ?  &  218 \\
18 & 7 & Proposition~\ref{prop:F4-16-18-20} & 2502 & 15 \\
20 & 8 & Proposition~\ref{prop:F4-16-18-20}  &0   & 4 \\
\noalign{\hrule height1pt}
\end{tabular}
}
\end{table}

\begin{table}[thbp]
\caption{Quaternary DT-optimal double circulant codes}
\label{Tab:F4-DCC}
\centering
\medskip
{\footnotesize
\begin{tabular}{c|l||c|l}
\noalign{\hrule height1pt}
$n$ & \multicolumn{1}{c||}{$r$} & $n$ &  \multicolumn{1}{c}{$r$} \\
\hline
4&$(1,\ww)$ &14&$(1,\ww,\ww^2,\ww,1,1,1)$ \\
6&$(1,\ww,1)$ &14&$(1,\ww^2,\ww,\ww^2,1,1,1)$ \\
8&$(1,1,1,0)$ &14&$(1,\ww,1,1,1,0,0)$ \\
8&$(1,\ww,1,0)$ &14&$(1,\ww^2,1,1,1,0,0)$ \\
8&$(1,1,\ww,0)$ &14&$(1,1,\ww,1,1,0,0)$ \\
8&$(1,\ww^2,\ww,0)$ &14&$(1,\ww,\ww,1,1,0,0)$ \\
8&$(1,\ww,1,1)$ &14&$(1,1,\ww^2,1,1,0,0)$ \\
8&$(1,\ww,\ww,1)$ &14&$(1,\ww^2,1,\ww,1,0,0)$ \\
10&$(1,\ww,1,\ww,0)$ &14&$(1,\ww,\ww^2,\ww,1,0,0)$ \\
10&$(1,\ww^2,1,\ww,0)$ &14&$(1,\ww^2,1,1,\ww,0,0)$ \\
12&$(1,\ww,1,1,0,0)$ &14&$(1,\ww^2,\ww^2,1,\ww,0,0)$ \\
12&$(1,\ww^2,1,1,0,0)$ &14&$(1,\ww^2,1,\ww,\ww,0,0)$ \\
12&$(1,1,1,\ww,0,0)$ &14&$(1,1,\ww,\ww^2,\ww,0,0)$ \\
12&$(1,1,\ww,0,1,0)$ &14&$(1,\ww,\ww,\ww^2,\ww,0,0)$ \\
12&$(1,\ww^2,\ww,0,1,0)$ &14&$(1,\ww^2,\ww^2,\ww^2,\ww,0,0)$ \\
12&$(1,1,\ww^2,0,1,0)$ &14&$(1,\ww^2,\ww,1,1,1,0)$ \\
12&$(1,\ww,\ww,1,1,0)$ &14&$(1,\ww^2,1,\ww,1,1,0)$ \\
12&$(1,\ww^2,\ww,1,1,0)$ &14&$(1,\ww,1,\ww^2,1,1,0)$ \\
12&$(1,\ww,\ww^2,1,1,0)$ &14&$(1,1,\ww,1,\ww,\ww,0)$ \\
12&$(1,\ww^2,\ww^2,1,1,0)$ &14&$(1,\ww,\ww^2,\ww,1,1,1)$ \\
12&$(1,\ww^2,\ww,\ww,1,0)$ &14&$(1,\ww^2,\ww,\ww^2,1,1,1)$ \\
12&$(1,\ww^2,\ww,1,1,1)$ &18&$(1,\ww^2,1,\ww,1,0,1,0,0)$ \\
12&$(1,\ww,\ww^2,\ww,1,1)$ &18&$(1,1,\ww^2,\ww,1,0,1,0,0)$ \\
14&$(1,\ww,1,1,1,0,0)$ &18&$(1,1,\ww,\ww^2,1,0,1,0,0)$ \\
14&$(1,\ww^2,1,1,1,0,0)$ &18&$(1,\ww^2,\ww,\ww,0,1,1,0,0)$ \\
14&$(1,1,\ww,1,1,0,0)$ &18&$(1,\ww^2,\ww^2,\ww,0,1,1,0,0)$ \\
14&$(1,\ww,\ww,1,1,0,0)$ &18&$(1,\ww^2,\ww,\ww,1,0,\ww,0,0)$ \\
14&$(1,1,\ww^2,1,1,0,0)$ &18&$(1,1,\ww^2,1,1,1,\ww,0,0)$ \\
14&$(1,\ww^2,1,\ww,1,0,0)$ &18&$(1,\ww,1,\ww,1,1,\ww,0,0)$ \\
14&$(1,\ww,\ww^2,\ww,1,0,0)$ &18&$(1,\ww^2,1,\ww,1,1,\ww,0,0)$ \\
14&$(1,\ww^2,1,1,\ww,0,0)$ &18&$(1,\ww^2,0,1,\ww,1,\ww,0,0)$ \\
14&$(1,\ww^2,\ww^2,1,\ww,0,0)$ &18&$(1,\ww,\ww,1,\ww,1,\ww,0,0)$ \\
14&$(1,\ww^2,1,\ww,\ww,0,0)$ &18&$(1,\ww^2,\ww^2,\ww,1,1,1,1,0)$ \\
14&$(1,1,\ww,\ww^2,\ww,0,0)$ &18&$(1,1,\ww^2,\ww^2,\ww,1,1,1,0)$ \\
14&$(1,\ww,\ww,\ww^2,\ww,0,0)$ &18&$(1,\ww,\ww,\ww,1,\ww,1,1,0)$ \\
14&$(1,\ww^2,\ww^2,\ww^2,\ww,0,0)$ &18&$(1,\ww,\ww,\ww^2,\ww,\ww,1,1,0)$ \\
14&$(1,\ww^2,\ww,1,1,1,0)$ &20&$(1,\ww^2,1,1,\ww,\ww,\ww^2,\ww,0,0)$ \\
14&$(1,\ww^2,1,\ww,1,1,0)$ &20&$(1,\ww,\ww^2,0,\ww^2,\ww,1,0,1,0)$ \\
14&$(1,\ww,1,\ww^2,1,1,0)$ &20&$(1,\ww,\ww^2,\ww,\ww,\ww,\ww^2,\ww,1,0)$ \\
14&$(1,1,\ww,1,\ww,\ww,0)$ &20&$(1,\ww^2,\ww,\ww^2,\ww,\ww^2,1,0,\ww,0)$ \\
\noalign{\hrule height1pt}
\end{tabular}
}
\end{table}

For the quaternary DT-optimal double circulant codes $\cC(r)$ 
of lengths $n \le 14$ and $n=18,20$, 
the rows $r$ are listed in Table~\ref{Tab:F4-DCC}.
The $7$ quaternary  DT-optimal double Toeplitz $[8,4,4]$ codes $\cT(t,a,b)$  are constructed as:
\begin{align*}
(t,a,b)=\,&
(0,(1,1,1), (1,\ww,1)),
\\&
(0,(1,1,1), (1,\ww^2,1)),
\\&
(0,(1,1,1), (1,1,\ww)),
\\&
(0,(1,1,1), (\ww^2,\ww^2,\ww)),
\\&
(0,(1,\ww,1), (\ww,1,\ww)),
\\&
(1,(1,1,0), (\ww^2,0,\ww)) \text{ and } 
\\&
(1,(\ww,1,0), (\ww^2,0,1)).
\end{align*}
The $2$ quaternary  DT-optimal double Toeplitz $[10,5,5]$ codes $\cT(t,a,b)$  are constructed as:
\begin{align*}
(t,a,b)=\,&
(0,(1,\ww,1,1), (1,\ww,\ww^2,\ww)) \text{ and } 
\\&
(1,(\ww,1,1,0), (\ww^2,1,0,\ww^2)).
\end{align*}
The remaining  quaternary  DT-optimal double Toeplitz codes 
and double circulant codes 
are available at
\url{https://www.math.is.tohoku.ac.jp/~mharada/DT}.

\subsection{Remark}
We end this paper with some remark on the classification.
Our classification reveals the existence of 
many DT-optimal double Toeplitz $[n,n/2]$ codes over $\FF_q$ that are inequivalent
to any of double circulant codes and double negacirculant codes.
There arises a natural question.
Let $d_{q,DC}(n)$ (resp.\ $d_{q,DN}(n)$)  denote the largest minimum weight among
all double circulant (resp.\ negacirculant) $[n,n/2]$ codes over $\FF_q$.
Then
it is worthwhile to determine whether there is a pair $(q,n)$ with 
\[
d_{q,DT}(n) > d_{q,DC}(n)
\]
or not for $q\in\{2,4\}$.
Also, it is worthwhile to determine whether there is an integer $n$ with 
\[
d_{3,DT}(n) > \max(d_{3,DC}(n),d_{3,DN}(n))
\]
or not.


\bigskip
\noindent
\textbf{Acknowledgments.}
This work was supported by JSPS KAKENHI Grant Number 23K25784.


\end{document}